\documentclass{amsart}

\usepackage{graphicx}
\usepackage{bbm}
\usepackage{amsmath}
\usepackage{amssymb}
\usepackage{amsthm}
\usepackage{xfrac}
\usepackage{enumerate} 
\usepackage{enumitem}
\usepackage{color}
\usepackage[titletoc]{appendix}
\usepackage{multirow}
\usepackage{fancyhdr}

\newcommand{\Z}{\Bbb Z}

\newcommand{\C}{\Bbb C}

\newtheorem{theorem}{Theorem}[section]
\newtheorem{lemma}[theorem]{Lemma}
\newtheorem{proposition}[theorem]{Proposition}
\newtheorem{definition}[theorem]{Definition}

\newtheorem{remark}[theorem]{Remark}

\newtheorem{conj}[theorem]{Conjecture}

\makeatletter
\newtheorem*{rep@theorem}{\rep@title}
\newcommand{\newreptheorem}[2]{%
	\newenvironment{rep#1}[1]{%
		\def\rep@title{#2 \ref{##1}}%
		\begin{rep@theorem}}%
		{\end{rep@theorem}}}
\makeatother
\newreptheorem{lemma}{Lemma}
\newreptheorem{definition}{Definition}
\newreptheorem{proposition}{Proposition}
\newreptheorem{theorem}{Theorem}

\numberwithin{equation}{section}

\newcommand{\Addresses}{{
  \bigskip
  \footnotesize

  D. Adamovi\' c, \textsc{ Department of Mathematics, University of Zagreb,  Croatia}\par\nopagebreak
  \textit{E-mail address:} \texttt{adamovic@math.hr}

  \medskip

  A. Kontrec, \textsc{Department of Mathematics, University of Zagreb,  Croatia}\par\nopagebreak
  \textit{E-mail address:} \texttt{ana.kontrec@math.hr}

}}

\begin{document}

\title[Irreducible modules for Bershadsky-Polyakov algebra]{Classification of irreducible modules for Bershadsky-Polyakov algebra at certain levels}

\author{Dra\v zen Adamovi\' c}


\author{Ana Kontrec}

\begin{abstract}
%
We study the representation theory of the Bershadsky-Polyakov algebra $\mathcal W_k = \mathcal{W}_k(sl_3,f_{\theta})$.  In particular, Zhu algebra of $\mathcal W_k$ is isomorphic to a certain quotient of the Smith algebra, after changing the Virasoro vector. 
We classify all modules in the category $\mathcal{O}$ for the Bershadsky-Polyakov algebra $\mathcal W_k$ for $k=-5/3, -9/4, -1,0$.  In the case $k=0$ we show that the Zhu algebra $A(\mathcal W_k)$ has $2$--dimensional indecomposable modules.
\end{abstract}
\date{\today}
\maketitle

\section{Introduction}

  Bershadsky-Polyakov  vertex algebra $\mathcal W_k = \mathcal{W}_k(sl_3,f_{\theta})$ is one of the simplest, but most important minimal affine $\mathcal W$--algebras. This vertex algebra is not of Lie type and its representation theory is very interesting. 
  
  \begin{itemize}
\item   In the physics literature, it appeared in the work of  M. Bershadsky \cite{Ber} and   A. M. Polyakov \cite{Pol}. 

\item A vertex algebraic framework for studying minimal affine W-algebras was presented in the work of V. Kac and M. Wakimoto \cite{KW} and V.  Kac,  S. Roan, M. Wakimoto \cite{KRW}. These papers also give a framework for studying the vertex algebra  $\mathcal W_k$.

\item Rationality of $\mathcal W_k$ for $k= p/2-3$, $p \ge 3$, $p$ odd was proved by T. Arakawa in \cite{A1}. This paper also contains many important new ideas for studying representation theory  of $\mathcal W_k$, which we will use in our paper. A particular importance of Arakawa's approach is in application of the theory of Smith algebras from \cite{S}.

\item Cosets of the Bershadsky-Polyakov algebra at half-integer levels were investigated in \cite{ACL}. 

\end{itemize}
  

There are two different approaches for studying the representation theory of $\mathcal W_k$. The first approach uses the representation theory of the affine vertex algebra $V_k(sl(3))$ and quantum Hamiltonian reduction \cite{KRW}, \cite{KW}.  This approach is very useful for constructing representations of $\mathcal W_k$. Moreover, using Arakawa-Moreau theory developed in papers \cite{ArM-1, ArM-2, ArM-3, ArM-4}  one sees that for admissible levels $k$,  $\mathcal{W}_k(sl_3,f_{\theta})$ is quasi-lisse, and it has finitely many irreducible modules in the category $\mathcal{O}$. But this approach is not sufficient for classifying irreducible representations.  Second approach uses explicit formulas for singular vectors in the universal vertex algebra $\mathcal W_k$, and projection of these singular vectors in Zhu's algebra for  $\mathcal{W}^k$. Although there are some approaches  to construct singular vectors in $\mathcal{W}^k$ using $\widehat{sl(3)}$--singular vector \cite{FGP}  it is still unclear how one constructs $\mathcal{W}^k$--singular vectors in general. Fortunately, in some cases one can find explicit formulas for singular vectors using straightforward calculation. We find such singular vectors for $k=-5/3, -9/4$ and for $k \in {\Bbb Z}$, $k \ge -1$. Using these singular vectors we are able to classify irreducible, highest weight representations.


As in the case of affine vertex algebras, the classification of irreducible representations is naturally connected with Zhu's algebra theory. If we choose the standard Virasoro element (denoted by $\omega$), the vertex algebra $\mathcal W_k$ is $\frac{1}{2} {\Z}_{\ge 0}$--graded, and Zhu's algebra $A_{\omega} (\mathcal W_k)$ is isomorphic to certain quotient of the polynomial algebra ${\C}[x,y]$. Let us illustrate this in the cases $k=-5/3, -9/4$.


\begin{table}[h!]
  \begin{center}
    \caption{Singular vectors of conformal weight $4$ and $3$ in $\mathcal{W}^k(sl_3,f_{\theta})$.}
    \label{tab:table1}
    \begin{tabular}{l|c|c}
       & \textbf{Formula for the singular vector in $\mathcal{W}^k(sl_3,f_{\theta})$} & \textbf{Projection in $A_{\omega}(\mathcal{W}^k)$}\\
      \hline
      \multirow{3}{*}{$ \Omega_4,\: k=-\frac{5}{3}$} & $-\frac{62}{9}L_{-2}^2\mathbbm{1} + \frac{14}{3}L_{-4}\mathbbm{1} - 18J_{-1}^4\mathbbm{1}  + 54J_{-2}J_{-1}^2\mathbbm{1}-$ & \\ 
      & $-130J_{-3}J_{-1}\mathbbm{1} +\frac{33}{2}J_{-2}^2\mathbbm{1} + 13J_{-4}\mathbbm{1} -12L_{-3}J_{-1}\mathbbm{1}+  $ &  $-18x^4+46x^2y-\frac{1}{2}x^2 -$\\ 
      & $46L_{-2}J_{-1}^2\mathbbm{1}- G^{+}_{-2}G^{-}_{-1}\mathbbm{1} + G^{+}_{-1}G^{-}_{-2}\mathbbm{1} -18J_{-1}G^{+}_{-1}G^{-}_{-1}\mathbbm{1} $ & $-\frac{62}{9}y^2 -\frac{10}{9}y$ \\ 
      \hline
      \multirow{2}{*}{$ \Omega_3, \: k=-\frac{9}{4}$} &  $\frac{3}{8}L_{-3}\mathbbm{1} +  J_{-1}^3\mathbbm{1} -3J_{-2}J_{-1}\mathbbm{1} + \frac{11}{4}J_{-3}\mathbbm{1}  -\frac{3}{2}L_{-2}J_{-1}\mathbbm{1}$ &  \\ 
      & $   + G^{+}_{-1}G^{-}_{-1}\mathbbm{1} $ & $x^3 - \frac{3}{2}xy - \frac{5}{8} x$ \\ 
    
    \end{tabular}
  \end{center}
\end{table}
But it turns out that in order to classify irreducible, highest weight representations, it is not enough to know the projection of singular vector in ${\C}[x,y]$.

We apply a different approach, which was used for  different levels in \cite{A1}. We   use a new Virasoro vector $\overline \omega = \omega + \frac{1}{2} DJ$,
and then  vertex algebras $\mathcal W^k$ and $\mathcal W_k$ become ${\Z}_{\ge 0}$--graded with respect to $L(0) =\overline \omega _1$. Then Zhu's algebras are not longer commutative, and it carries more information on the representation theory.
Zhu algebra for $\mathcal{W}^k$ is realized as a quotient of the Smith algebra from \cite{S}.    Let $ g(x,y) \in \mathbb{C}[x,y] $ be an arbitrary polynomial. Associative algebra \textbf{ $R(g)$ of Smith type} is generated by $\left\{E,F,X,Y\right\}$  such that $Y$  is a central element and the following relations hold:
	\begin{equation*}
	XE-EX = E, \: XF-FX=-F, \:EF-FE=g(X,Y).
	\end{equation*}

\begin{table}[h!]
  \begin{center}
    \caption{Singular vectors of conformal weight $4$ and $3$ in $\mathcal{W}^k(sl_3,f_{\theta})$ with a new Virasoro field $  L (z) := T(z) + \frac{1}{2} D  J(z)$.}
    \label{tab:table1}
    \begin{tabular}{l|c|c}
       & \textbf{Formula for the singular vector in $\mathcal{W}^k(sl_3,f_{\theta})$} & \textbf{Relation in $A_{\overline \omega}(\mathcal{W}_k)$ }\\
      \hline
      \multirow{5}{*}{$\overline \Omega_4,\: k=-\frac{5}{3}$} & $-\frac{62}{9} L(-2)^2\mathbbm{1} + \frac{14}{3}  L(-4)\mathbbm{1} - 18J(-1)^4\mathbbm{1}  + 31J(-2)J(-1)^2\mathbbm{1}  -$ & \\ 
      & $- 118J(-3)J(-1)\mathbbm{1}+ \frac{133}{9}J(-2)^2\mathbbm{1} - \frac{8}{9}J(-4)\mathbbm{1} + $ &  $ [G^+]^2([\overline \omega] + \frac{1}{9}) = 0$\\ 
      & $\frac{62}{9}  L(-2)J(-2)\mathbbm{1} - 12  L(-3)J(-1)\mathbbm{1}+ 46  L(-2)J(-1)^2\mathbbm{1}  - $ & \\ 
      & $-G^{+}(-2)G^{-}(-2)\mathbbm{1} + G^{+}(-1)G^{-}(-3)\mathbbm{1} -$ & \\ 
       & $ -18J(-1)G^{+}(-1)G^{-}(-2)\mathbbm{1}$ & \\ 
      \hline
      \multirow{2}{*}{$\overline \Omega_3, \: k=-\frac{9}{4}$} & $\frac{3}{8} L(-3)\mathbbm{1} +  J(-1)^3\mathbbm{1} -\frac{9}{4}J(-2)J(-1)\mathbbm{1}  + \frac{19}{8}J(-3)\mathbbm{1}$ &  \\ 
      & $  -\frac{3}{2} L(-2)J(-1)\mathbbm{1}+  G^{+}(-1)G^{-}(-2)\mathbbm{1} $ & $[G^+]([\overline \omega] + \frac{1}{2}) = 0$ \\ 
    
    \end{tabular}
  \end{center}
\end{table}

In the case $k =-5/3$, we show that $\mathcal{W}_k$ can be realized as a $\mathbb{Z}_3$-orbifold (fixed points subalgebra) of the Weyl vertex algebra $W$ (cf. Proposition \ref{orbifold}). In this case $\mathcal{W}_k$ has exactly $9$ irreducible,  highest weight modules (cf. Proposition \ref{klas_5_3}, Theorem  \ref{class-o-1}). Of these $9$ modules, $6$ have finite-dimensional weight spaces for $L(0)$, while others have infinite-dimensional weight spaces.


In the case $k =-9/4$, $\mathcal{W}_k$ is an important example of a logarithmic vertex algebra. We prove that $\mathcal{W}_k$ has exactly $6$  irreducible, highest weight  modules (cf. Proposition \ref{klas_9_4}), of which $3$ have have finite-dimensional weight spaces for $L(0)$.
 

Bershadsky-Polyakov vertex algebra $\mathcal{W}_k$ is a part of a series of vertex algebras which can be realized using vertex algebras from logarithmic conformal field theory, the so-called $B_p$--algebras. In particular, for $k=-9/4$ it holds that  $\mathcal{W}_k \cong B_4. $ We construct an uncountable family of weight modules for $\mathcal{W}_k$ outside of category $\mathcal{O}$ (cf. Theorem \ref{M_r_modul}).
 
We construct a family of singular vectors in  $\mathcal{W}^k$ which generalizes Arakawa's formulas for singular vectors for $k= p/2-3$, $p \ge 3$, $p$ odd (cf. \cite{A1}).  
Namely,	vectors $$G^+(-1)^{n}\mathbbm{1}, \; G^-(-2)^{n}\mathbbm{1}$$ are singular for $n = k+2$, where $k \in \mathbb{Z}$, $k \geq -1$ (cf. Lemma \ref{integral-sing}).
	
For $k=-1$ and $k=0$ we classify all modules in the category $\mathcal{O}$. In the case $k=0$, we give an explicit realization of the vertex algebra $\mathcal{W}_k$ and its modules as certain irrational subalgebras of lattice VOAs. We prove (cf. Theorem \ref{class-k0}):
\begin{theorem}
\begin{itemize}
    \item[(1)] The vertex  algebra  $\mathcal{W}_0$  has two families of  irreducible  highest weight modules
	$$ \{ L(x, x^2 + (i-1) x) \vert x \in {\C}, i = 0,1 \}.$$ 
\item[(2)]  There exists a $ \Z_{\ge 0}$--graded $\mathcal W_0$--module $M$ such that $M_{top}$ is a $2$--dimensional indecomposable module for the Smith algebra $R(g)$.
\end{itemize}
\end{theorem}

Representation theory of the Bershadsky-Polyakov vertex algebra $\mathcal{W}_k$ when $ k \in {\Z}_{\ge 1}$ is much more difficult, and  it will be studied in our forthcoming papers.


\subsection*{ Setup} 
\begin{itemize}
\item  The universal Bershadsky-Polyakov algebra of level $k$  will be denoted with $\mathcal{W}^k(sl_3,f_{\theta})$ or $\mathcal W^k$;
\item its unique simple quotient with $\mathcal{W}_k(sl_3,f_{\theta})$ or  $\mathcal W_k$. 
\item Modes of fields with respect to Virasoro vector $\omega$: $J_n$, $L_n$, $G^{\pm}_n$.
\item $L_{x,y}$ denotes the irreducible highest weight representation with highest weight $(x,y)$  with respect to $(J_0, L_0)$ and highest weight vector $v_{x,y}$.
\item $L_{x,y} ^*$ denotes the module contragredient to $L_{x,y}$ with respect to the Virasoro algebra generated by $L_n, n \in \mathbb{Z} $
\item Modes of fields with respect to Virasoro vector $\overline \omega=\omega + \frac{1}{2} DJ$: $$J(n)= J_n, \ L(n) = L_n - \frac{n+1}{2} J_n,  \  G^+ (n) = G^+_n,  \ G^-  (n) = G^- _ {n+1}.  $$
\item $L(x,y) $ denotes the irreducible highest weight representation with highest weight $(x,y)$  with respect to $(J(0), L(0) )$ and highest weight vector $v(x,y)$.
\item We have $L(x,y) =L_{x, y+x/2}$.
\item The Zhu algebra associated to the vertex operator algebra $V$ with the Virasoro vector $\omega$  will be denoted with $A_{\omega}(V)$. 
\item The Smith algebra corresponding to the polynomial $g(x,y) \in \mathbb{C}[x,y]$ is denoted with $R(g)$.
\end{itemize}

\section{Preliminaries}

In this section we briefly recall definitions of certain types of modules for vertex operator algebras, as well as examples of VOAs that will be used in later sections. As the Bershadsky-Polyakov algebra is $\frac{1}{2}\mathbb{Z}$-graded, we also define contragredient modules and Zhu algebra for $\frac{1}{2}\mathbb{Z}$-graded vertex operator algebras.

\subsection{Vertex algebras and its modules}

Following \cite{DLM2}, we will say that $V$ is a $\frac{1}{2}\mathbb{Z}$-graded vertex operator algebra if $V$ satisfies all the axioms for the VOA except that $V$ is $\frac{1}{2}\mathbb{Z}$-graded instead of $\mathbb{Z}$-graded.


Let $V$  be a $\frac{1}{2}\mathbb{Z}$-graded VOA. A \textbf{weak $V$--module}  is the pair $(M, Y)$, where $M$ is a complex $\mathbb{C}$-graded vector space and $Y$ is the operator $$ Y: V \longrightarrow (\text{End} \, M)[[z,z^{-1}]], \quad Y(v,z)= \sum _{n \in \mathbb{Z}} v_{n}z^{-n-1} $$ satisfying the following conditions for $u,v \in V$, $w \in M$:
\begin{enumerate}
		\item $Y(\mathbbm{1}, z)= Id_M,$
		
		\item $v_nw = 0$ for $n \in \mathbb{Z}$ sufficiently large,
		
		\item The Jacobi identity holds:
		\begin {equation*}
		\begin {aligned}
		z^{-1}_0 \delta \left ( \frac{z_1 - z_2}{z_0} \right)Y_M(u, z_1)Y_M(v, z_2)w - z^{-1}_0 \delta \left ( \frac{-z_2 + z_1}{z_0} \right)Y_M(v, z_2)Y_M(u, z_1)w = 
		\\ = z^{-1}_2 \delta \left ( \frac{z_1 - z_0}{z_2} \right)Y_M(Y(u, z_0)v,z_2)w.
		\end {aligned}
		\end {equation*}
		
		\item Setting $Y(\omega, z) = \sum _{n \in \mathbb{Z}} L_nz^{-n-2}$, it holds that
		$$ [L_m, L_n ] = (m-n)L_{m+n} + \frac{m^3-m}{12} c\delta _{m+n,0}, \; (m,n \in \mathbb{Z}),  $$
		$$ \frac{d}{dz}Y(v,z) = Y(L_{-1}v, z). $$
\end{enumerate}

\medskip

An \textbf{ordinary $V$-module} $(M, Y)$ is a weak $V$-module   $ M = \bigoplus _{h \in \mathbb{C} }M(h), $ where $M(h) = \{ w \in M: L_0w = hw \}$, such that $\forall h \in \mathbb{C}$, $\text{dim } M(h)<\infty$ and $M(h+n) = 0$ for $n \in \mathbb{Z}$ sufficiently small.

\medskip


\subsection{Contragredient modules}

The notion of contragredient modules for VOAs was originally introduced in \cite{FHL} and further generalized to $\mathbb{Q}$-graded VOAs in \cite{DLM2}. Here we only state the results for $\frac{1}{2}\mathbb{Z}$-graded VOAs (see \cite{DLM2}, \cite{Li2}  for the more general theory of contragredient modules for $\mathbb{Q}$-graded VOAs).

\smallskip

Let $V$ be a $\frac{1}{2}\mathbb{Z}$-graded VOA and let  $ M = \bigoplus _{h \in \mathbb{C} }M(h)$ be a  ordinary $V$-module such that $ \text{dim } M(h)<\infty$, $\forall h \in \mathbb{C}$. Let  $ M^* = \bigoplus _{h \in \mathbb{C} }M^*(h)$ be the graded dual space of $M$. Define the adjoint vertex operators $Y^*(v,z)$ for $v \in V$ as  $$ Y^*: V  \longrightarrow (End \, M^*)[[z,z^{-1}]],  \quad  Y^*(v,z) = \sum _{n \in \mathbb{Z}} v^*_{n}z^{-n-1}$$ satisfying the condition $$ \langle Y^*(v,z)w^*, w \rangle = \langle w^*, Y(e^{zL(1)}e^{\pi iL(0)}z^{-2L(0)}v, z^{-1})w \rangle $$ for $v \in V$, $w^* \in M^*$, $w \in M$. Here $\langle \cdot , \cdot \rangle$ denotes the natural pairing between $V^*$ and $V$.

\begin{proposition}[\cite{DLM2}, Proposition 3.7., Remark 3.8.]
    Let $V$ be a $\frac{1}{2}\mathbb{Z}$-graded vertex operator algebra. Then $M^*$ has the structure of a $V$-module.
\end{proposition}
We will call $M^*$ the $V$-module \textbf{contragredient} to $M$.

\smallskip

Let $V$ be a VOA and $M$ a $V$-module. In \cite{FHL} it was proved that the double contragredient module $(M^*)^*$ is isomorphic to $M$, and hence $M$ is irreducible if and only if $M^*$ is irreducible (cf. \cite{FHL}, Proposition 5.3.2.) If $V$ is a $\mathbb{Q}$-graded VOA, the following relation holds (cf. \cite{Li2}, Remark 2.5.): 
\begin{equation}\label{double}
 (Y^*)^*(v,z) = Y(\tau(v),z),    
\end{equation}
where $\tau = e^{2\pi i L(0)}$ is an automorphism of $V$.

Recall that the automorphism group $Aut(V)$ acts on the set of irreducible modules by $M \mapsto M \circ g$, for $g \in Aut(V)$ (cf. \cite{DM}).

Since by   (\ref{double}), the double contragredient module $(M^*)^*$is isomorphic to the module $M \circ \tau$, we conclude that   $(M^*)^*$ is irreducible if and only if $M$ is irreducible.  This implies that $M^*$ is irreducible if and only if $M$ is irreducible.



\subsection{Zhu algebra}

Let $V = \bigoplus _{r \in \frac{1}{2}\mathbb{Z} }V(r)$ be a $\frac{1}{2}\mathbb{Z}$-graded VOA, and let $\text{deg }a = r$ for $a \in V(r)$. Let 
	$$ V^{0} = \bigoplus _{r \in \mathbb{Z} }V(r), \quad V^{1} = \bigoplus _{r \in \frac{1}{2} + \mathbb{Z} }V(r). $$
	Define bilinear mappings $*: V \times V \longrightarrow V$, $\circ: V \times V \longrightarrow V $:
	$$ a*b =
	\begin{cases}
	Res_z \left(Y(a,z)  \frac{(1+z)^{\text{deg} \, a}}{z}b \right), & \text{for } a,b \in V^0 \\
	0, & \text{for } a,b \in V^1
	\end{cases}$$
	
	$$ a \circ b =
	\begin{cases}
	Res_z \left(Y(a,z)  \frac{(1+z)^{\text{deg} \, a}}{z^2}b \right), & \text{for } a \in V^0 \\
	Res_z \left(Y(a,z)  \frac{(1+z)^{\text{deg} \, a - \frac{1}{2}}}{z}b \right), & \text{for } a \in V^1,
	\end{cases}$$
	for homogenuous $a,b$.
	
Let $O(V) \subset V$ be the linear span of the elements $a \circ b$. The quotient space
	$$ A(V) = \frac{V}{O(V)} $$
is an associative algebra called the \textbf{Zhu algebra} of the VOA $V$.

\smallskip

In the case  $V^{1} = \{0\}$, we get the usual definition of Zhu's algebra for vertex operator algebras \cite{Z}. 

\smallskip

For homogenuous $a,b \in V $, the following formula holds (\cite{Z}, Lemma 2.1.3.):
\begin{equation} \label{eq: zhu_komutator}
	a*b - b*a = Res_z(1+z)^{\text{deg} \, a -1}Y(a,z)b + O(V).
\end{equation}

\subsection{The Weyl vertex algebra  } 

The Weyl vertex algebra $W$ is the universal vertex algebra  generated by the  even fields $a ^+$ and $a^-$  and the following  non-trivial $\lambda$--bracket:
$$ [a ^+ _{\lambda} a^-  ]= 1. $$
  The vertex algebra $W$    has the structure  of the irreducible  level one  module for the Lie  algebra  with generators
$ \{ K, a^{\pm}_n \ \vert  \ n \in {\Z} \}$ and commutation relations:
$$[a^+_r, a ^-_s] = \delta_{r+s+1,0} K , \quad  [a ^{\pm}_r , a ^{\pm}_s]=0 \quad (r,s \in  \Z, ), $$
where $K$ is central element. The fields $a^{\pm} $  acts on $W$ as the following Laurent series
$$a ^{\pm} (z)  = \sum_{ n \in {\Z} } a^{\pm}_n  z^{-n-1}.$$

\subsection{ Clifford vertex algebra} \label{clif} 

Clifford vertex algebra $F$ is the universal vertex algebra  generated by the odd fields  $\Psi^{+}$ and $\Psi^{-}$ , and the following $\lambda$--brackets: $$ [\  \Psi^+  _{\lambda} \Psi^- \  ] =  1, \quad [\  \Psi^{\pm}  _{\lambda} \Psi^{\pm} \  ] =  0. $$ 

The vertex algebra $ F$ has the structure of an irreducible module for the Clifford algebra $\textit{Cl}(A)$ associated to the vector superspace $A =\mathbb{C}\Psi^{+} \oplus \mathbb{C}\Psi^{-}$ with generators 
 $\Psi^{\pm}  (r),  \ \ r \in \tfrac{1}{2} + {\Z}$  and commutation relations
$$\{\Psi^{+} (r), \Psi^{-} (s)  \} = \delta_{r+s,0}, \quad \{\Psi^{\pm} (r), \Psi^{\pm} (s)  \} = 0. $$ 

As a vector space,
$$F   = \bigwedge \left( \Psi^{\pm} (1/2-n),   \ n  \in {\Z}_{>0}  \right). $$
Let  $\alpha = :\Psi^+ \Psi^- :$. The field $\alpha(z)$ generates a  Heisenberg vertex subalgebra of $F$, which we denote with $M(1)$. Vertex algebras  $M(1) $ and $F$ have the following Virasoro vector of central charge $c=1$:
$$ \omega_F = \frac{1}{2} :\alpha \alpha:,  $$ 
so that $\Psi^{\pm}$ are primary fields of conformal weight $1/2$.

If we choose the Virasoro vector $$\omega_{c=-2} =  \frac{1}{2}   ( :\alpha \alpha: + D\alpha)$$
of conformal weight $c=-2$, then $\Psi^{+}$ and $\Psi^-$  are primary fields of conformal weights $0$ and  $1$ respectively.

\subsection{ Symplectic fermions}\label{sf}


 Symplectic fermions $\mathcal A (1) $ are the universal vertex algebra  generated by the odd fields $b$ and $c$, and the following $\lambda$--brackets:
$$ [ b  _{\lambda} c ] =  \lambda, \quad [ b  _{\lambda} b ] = [ c _{\lambda} c ] = 0. $$
The vertex algebra $\mathcal A (1) $ has the structure of an irreducible module for the  Lie superalgebra with generators
$b(n), c(n), n \in \mathbb{Z}$  and commutation relations
$$\{b (n), c(m) \} = n \delta_{n+m,0}. $$
All other super-commutators are equal to 0. As a vector space,
$$\mathcal A(1)  \cong \bigwedge \left( b(-n), c(-n), \ n \in {\Z}_{>0}  \right). $$
 
$\mathcal A (1) $ has the following Virasoro vector of central charge  $c=-2$: $$\omega_{\mathcal A(1)}  = :b c:. $$

There is a conformal embedding of  $\mathcal A (1) $  into $F$ such that
$$ b  = - D\Psi^+, \ c= \Psi^-,  \omega_{\mathcal A(1)}  = \omega_{c=-2}.$$

\section{Bershadsky-Polyakov algebra $\mathcal{W}_k(sl_3,f_{\theta})$}

Bershadsky-Polyakov vertex algebra $\mathcal{W}^k (=\mathcal{W}^k(sl_3,f_{\theta}))$ is the minimal affine $\mathcal{W}$-algebra associated to the minimal nilpotent element $f_{\theta}$. The algebra $\mathcal{W}^k$ is generated by four fields $T, J,G^+,G^-$, of conformal weights $2, 1, \frac{3}{2}, \frac{3}{2}$ and is a $\frac{1}{2}\mathbb{Z}$-graded VOA.
\begin{definition}
	Universal Bershadsky-Polyakov vertex algebra $\mathcal{W}^k $  is the vertex algebra generated by fields $T ,J,G^+,G^-$, which satisfy the following relations: 
	\begin{align*}
    J(x)J(y) &\sim \frac{2k+3}{3}(z-w)^{-2}, \enskip  G^{\pm}(z)G^{\pm}(w) \sim0, &\\
    J(z)G^{\pm}(w) &\sim \pm G^{\pm}(w)(z-w)^{-1},&\\
	T(z)T(w) &\sim - \frac{c_k}{2}(z-w)^{-4}+2T(w)(z-w)^{-2}+DT(w)(z-w)^{-1},   &\\
	T(z)G^{\pm}(w) &\sim \frac{3}{2} G^{\pm}(w)(z-w)^{-2}+ DG^{\pm}(w)(z-w)^{-1}, &\\
	T(z)J(w) &\sim  J(w)(z-w)^{-2}+DJ(w)(z-w)^{-1}, &\\
	G^+(z)G^-(w) &\sim (k+1)(2k+3)(z-w)^{-3}+3(k+1)J(w)(z-w)^{-2}+ &\\
	    &+ (3:J(w)J(w): + \frac{3(k+1)}{2}DJ(w)-(k+3)T(w))(z-w)^{-1}, &
\end{align*}
where $c_k = -\frac{(3k +1)(2k +3)}{k+3}$.
\end{definition}

Vertex algebra $\mathcal{W}^k$ is called the universal Bershadsky-Polyakov vertex algebra of level $k$. For $k \ne -3$ , $\mathcal{W}^k$ has a unique simple quotient which is denoted by $\mathcal{W}_k$. 

Let
\begin{align*}
 T(z) &= \sum _{n \in \mathbb{Z}} L_nz^{-n-2} &\\
 J(z) &=  \sum _{n \in \mathbb{Z}} J_nz^{-n-1}, &\\ 
 G^+(z) &=  \sum _{n \in \mathbb{Z}} G^+_nz^{-n-1}, &\\ 
 G^-(z) &=  \sum _{n \in \mathbb{Z}} G^-_nz^{-n-1}.   &
\end{align*}

The following commutation relations hold:
\begin{align*}
    [J_m, J_n ] &= \frac{2k+3}{3} m\delta _{m+n,0} , \quad [J_m, G^{\pm}_n] = \pm G^{\pm}_{m+n},   &\\ 
    [L_m, J_n] &= - nJ_{m+n}, &\\ 
    [L_m, G^{\pm}_n] &= (\frac{1}{2}m-n+\frac{1}{2})G^{\pm}_{m+n}, &\\  [G^+_m,G^-_n] &= 3(J^2)_{m+n-1} + \frac{3}{2}(k+1)(m-n)J_{m+n-1}  -  (k+3)L_{m+n-1} + &\\ &+ \frac{(k+1)(2k+3)(m-1)m}{2} \delta _{m+n,1}. &
\end{align*}

\medskip   
 
By applying results from \cite{KW} we see that for every $(x,y) \in \mathbb{C}^2$ there exists an irreducible representation $L_{x,y}$ of $\mathcal{W}^k$ generated by a highest weight vector $v_{x,y}$ such that
$$ J_{0}v_{x,y} = xv_{x,y}, \quad J_{n}v_{x,y} = 0 \; \text{for} \: n>0, $$
$$ L_{0}v_{x,y} = yv_{x,y}, \quad L_{n}v_{x,y} = 0 \; \text{for} \: n>0, $$
$$ G^{\pm}_{n}v_{x,y} = 0 \; \text{for} \: n\geq 1. $$

\medskip

Let $A_{\omega}(V)$ be the Zhu algebra associated to the VOA $V$ with the Virasoro vector $\omega$, and let $[v]$ be the image of $v \in V $ under the mapping $V \mapsto A_{\omega}(V)$. 

For the Zhu algebra $A_{\omega}(\mathcal{W}^k)$ it holds that:
\begin{proposition}
	There exists a homomorphism  $\Phi :  {\Bbb C}[x,y] \rightarrow A_{\omega}(\mathcal{W}^k)$ such that
	$$ \Phi ( x ) = [J], \ \Phi(y) = [T]. $$ 
\end{proposition}
\begin{proof}
	Since the fields $G^+, G^-, J, T$ are strong generators for $\mathcal{W}^k$, the Zhu algebra $A_{\omega}(\mathcal{W}^k)$ is generated by $[G^+], [G^-], [J], [T]$. But, since
	\begin{align*}
	[G^{\pm}_{-1} \mathbbm{1}] \circ [\mathbbm{1}]  &= Res_z G^{\pm} (z)\frac{(1+z)^{\text{deg} \, G^{\pm}_{-1} - \frac{1}{2}}}{z} \mathbbm{1} & \\
	&= Res_z(\ \frac{1}{z} + 1)(\sum G^{\pm}_{n}z^{-n-1})\mathbbm{1}  & \\
	&= G^{\pm}_{-1}\mathbbm{1} = G^{\pm} \in O(V), &
	\end{align*}
	we have that $[G^{\pm}] = 0$ in $A_{\omega}(\mathcal{W}^k)$. 
	
As $[T]$ is in the center of $A_{\omega}(\mathcal{W}^k)$, we conclude that $A_{\omega}(\mathcal{W}^k)$ is a commutative algebra and hence there is a homomorphism from the symmetric algebra in two variables  into $A_{\omega}(\mathcal{W}^k)$.
\end{proof}

\begin{remark}
It can be shown that the homomorphism $\Phi :  {\Bbb C}[x,y] \rightarrow A_{\omega}(\mathcal{W}^k)$ is in fact an isomorphism, i.e. that $A_{\omega}(\mathcal{W}^k) \cong {\Bbb C}[x,y]$.
\end{remark}

 
\medskip 
 
Let $L_{x,y} ^*$ denote the module contragredient to $L_{x,y}$ with respect to the Virasoro algebra generated by $L_n, n \in \mathbb{Z}$ (cf.  Section 2). The following lemma illustrates the symmetry property of contragredient modules.
\begin{lemma} \label{contragredient}
\item[(1)] The contragredient module $L_{x,y} ^*$ is isomorphic to $L_{-x, y}$.  
\item[(2)] Assume that $L_{x,y}$ is a  $\mathcal{W}_k$--module. Then $L_{-x, y}$ is also a  $\mathcal{W}_k$--module.
\end{lemma}


\section{Smith algebra and Zhu's algebra $A_{\overline \omega}(\mathcal{W}^k)$}

The class of associative algebras $R(f)$ parametrized by an arbitrary polynomial $f(x) \in \mathbb{C}[x]$,  with generators $\left\{A,B,H\right\}$ satisfying relations
\begin{equation*}
HA-AH = A, \: HB-BH=-B, \:AB-BA=f(H).
\end{equation*}
were introduced by S. P. Smith in \cite{S}.  

These algebras were used by T. Arakawa in the paper \cite{A1} to prove rationality of $\mathcal{W}_k(sl_3,f_{\theta})$ for $k= p/2-3$, $p \ge 3$, $p$ odd. In this section we recall some results from this paper which we will use to classify $\mathcal{W}_k$-modules, and prove that the Zhu algebra associated to $\mathcal{W}_k$ is a quotient of the Smith algebra. First we expand the original definition of Smith algebras $R(f)$ by adding a central element.

\medskip

\begin{definition}\label{smith_type}
	Let $ g(x,y) \in \mathbb{C}[x,y] $ be an arbitrary polynomial. Associative algebra \textbf{ $R(g)$ of Smith type} is generated by $\left\{E,F,X,Y\right\}$  such that $Y$  is a central element and the following relations hold:
	\begin{equation*}
	XE-EX = E, \: XF-FX=-F, \:EF-FE=g(X,Y).
	\end{equation*}
\end{definition}

\subsection{Structure of the Zhu algebra $A_{\overline \omega}(\mathcal{W}^k)$}




Now we consider the Zhu algebra $A_{\overline \omega}(\mathcal{W}^k)$ associated to the Bershadsky-Polyakov algebra $\mathcal{W}^k$ with a  new Virasoro field  $$  L (z)  := T(z) + \frac{1}{2} D  J(z).$$

 Then $ \overline \omega = \omega + \frac{1}{2} DJ$ is a conformal vector $ \overline \omega_{n+1} =  L(n) $ with central charge $$\overline c_k = -\frac{4(k +1)(2k +3)}{k+3}.$$ The fields $J, G^+, G^-$ have conformal weights $1,1,2$ respectively.
Set $ J(n) = J_n, \, G^+ (n) = G^+ _n, \, G^-  (n) = G^-  _ {n+1}$. We have
 $$  L(z) = \sum_{n \in {\Z} }  L(n) z^{-n-2}, \; G^+ (z) = \sum_{n \in {\Z} }  G^+ (n) z ^{-n-1}, \;  G^- (z) = \sum_{n \in {\Z} } G^- (n) z^{-n-2}. $$
 
 
This defines a  $\mathbb{Z} _{\ge 0}$-gradation on $\mathcal{W}^k$.

\begin{proposition} \label{smith-zhu}
	Zhu algebra $A_{\overline \omega}(\mathcal{W}^k)$ is a quotient of the Smith algebra $R(g)$ for $g(x,y) = -(3x^2 - (2k+3)x - (k+3)y)$.
\end{proposition}
\begin{proof}
Let $$ E=[G^+], \ F=-[G^- ], \   X=[J ],  \ Y =[\overline \omega], $$ where $[G^+], [G^- ], [J ], [\overline \omega]$ are the generators of the Zhu algebra $A_{\overline \omega}(\mathcal{W}^k)$. Using formula (\ref{eq: zhu_komutator}), we have:
	\begin{flalign*}
	X*E - E*X 
	&= Res_z(\sum J(n)z^{-n-1})G^+(-1)\mathbbm{1} + O(V) &\\ 
	&= [J(0)G^+(-1)\mathbbm{1}]=[G^+(-1)\mathbbm{1}]= E,   &
	\end{flalign*}
	
	\begin{flalign*}
	X*F - F*X 
	&=- Res_z(\sum J(n)z^{-n-1})G^-(-2)\mathbbm{1} + O(V) &\\ 
	&= -[J(0)G^-(-2)\mathbbm{1}]=[G^-(-2)\mathbbm{1}]= -F,   &
	\end{flalign*}
	
	\begin{flalign*}
	X* Y -  Y*X 
	&= Res_z(\sum J(n)z^{-n-1}) L(-2)\mathbbm{1} + O(V) &\\ 
	&= [J(0) L(-2)\mathbbm{1}] =0,  &
	\end{flalign*}

	\begin{flalign*}
	E*F - F*E 
	&=- Res_z(\sum G^+(n)z^{-n-1})G^-(-2)\mathbbm{1} + O(V) &\\ 
	&= -[G^+(0)G^-(-2)\mathbbm{1}] &\\
	& =-[3J^2(-2)\mathbbm{1} + (2k+3)J(-2)\mathbbm{1} - (k+3) L(-2)\mathbbm{1}]   &\\
	& = -(3X^2 - (2k+3)X - (k+3)Y) = g(X, Y),&
	\end{flalign*}
	
	\begin{flalign*}
	E*Y - Y*E 
	&= Res_z(\sum G^+(n)z^{-n-1})L(-2)\mathbbm{1} + O(V) &\\ 
	&= [G^+(0) L(-2)\mathbbm{1}] =0, &
	\end{flalign*}
	
	\begin{flalign*}
	F*Y - Y*F 
	&= -Res_z(1+z)(\sum G^-(n)z^{-n-2}) L(-2)\mathbbm{1} + O(V) &\\ 
	&= -[G^-(-1)L(-2)\mathbbm{1}] - [G^-(0) L(-2)\mathbbm{1}] &\\
	&= -[L(-1)G^-(-2)\mathbbm{1}] - 2[G^-(-2)\mathbbm{1}] &\\
	&= 2[G^-(-2)\mathbbm{1}] - 2[G^-(-2)\mathbbm{1}] = 0. &
	\end{flalign*}
	
 As the generators $E,F,X,Y$ satisfy the defining relations for the Smith algebra, we conclude that there is a homomorphism from the Smith algebra $R(g)$ into the Zhu algebra $A_{\overline \omega}(\mathcal{W}^k)$.

\end{proof}

\subsection{Modules for the Zhu algebras $A_{\omega}(\mathcal{W}^k)$ and $A_{\overline \omega}(\mathcal{W}^k)$}

In the previous sections we showed that the Zhu algebra $A_{\omega}(\mathcal{W}^k)$ is isomorphic to the algebra of polynomials in two variables ${\Bbb C}[x,y]$, while the Zhu algebra $A_{\overline \omega}(\mathcal{W}^k)$ (corresponding to the Bershadsky-Polyakov algebra with a shifted Virasoro vector $\overline{\omega}$) is a quotient of the Smith algebra $R(g)$ for a certain polynomial $g$.

Now we investigate the relationship between modules for these two Zhu algebras.
    
\medskip

Consider    the irreducible highest weight module $L(x,y) $    for $\mathcal{W}^k$, generated by a highest weight vector $ v(x,y):= v_{x,y}$ with highest weight $(x,y+x/2) \in \mathbb{C}^2$. 
 For $n \ge 0$ we have:
 $$  L(n) v(x,y)  = ( L_n - \frac{n+1}{2} J_n ) v(x,y)  =   y \delta_{n,0} v(x,y). $$
Clearly we have: $$ L(x,y) = L_{x, y+ x/2}.$$
 We get:
\begin{proposition} \label{rep_exs}
	For every $(x,y) \in \mathbb{C}^2$ there exists an irreducible representation $L(x,y)$ of $\mathcal{W}^k$ generated by a highest weight vector $v(x,y)$ such that
	$$ J(0)v(x,y)  = xv(x,y),  \quad J(n)v(x,y)  = 0 \; \text{for} \: n>0, $$
	$$ L(0)v(x,y)  = yv(x,y) , \quad  L(n)v(x,y)  = 0 \; \text{for} \: n>0, $$
	$$G^{-}(n-1)v(x,y)   =  G^{+}(n)v(x,y)  = 0 \; \text{for} \: n\geq 1. $$
\end{proposition}



\begin{lemma} \label{simetrija-1}
$ L(x,y)$ is a $\mathcal W_k$--module if and only if $L(-x, y + x)$ is a  $\mathcal W_k$--module.
\end{lemma}     
\begin{proof}      
 The proof follows using the symmetry property of contragredient  modules (Lemma \ref{contragredient}):
 $$ L(x,y) ^* = L_{x, y + x/2} ^* = L_{-x, y + x/2} = L(-x, y + x). $$
 \end{proof}

\bigskip

Define:
$$ L(x,y)_{\text{top}} = \left\lbrace v \in L(x,y) :  L(0)v = y v \right\rbrace. $$

Now let $$g(x,y) = -(3x^2 - (2k+3)x - (k+3)y)\in \mathbb{C}[x,y]$$ and define polynomials $h_i(x,y)$, for $i\in \mathbb{N}$ (cf. \cite{A1}) as
\begin{align*} 
h_i(x,y) &= \frac{1}{i}(g(x,y) + g(x+1,y) + ... + g(x+i-1,y)) &\\
&= -i^2+ki-3xi+3i-3x^2-k+2kx+6x+ky+3y-2.&
\end{align*} 

\smallskip

The following properties were proven in \cite{A1}:

\begin{proposition} [\cite{A1}, Proposition 2.2.] \label{l_top}
Let $L(x,y)_{\text{top}}$ and $h_i(x,y)$, for $i\in \mathbb{N}$ be as above. Then:
\begin{itemize}
    \item[(i)] $L(x,y)_{\text{top}}$ is spanned (as a vector space) by vectors $\left\{(G^+(0))^i v(x,y)  \; 0 \leq i  \right\}$,
    \item[(ii)] If the space $ L(x,y)_{\text{top}}$ is $n$-dimensional, then $h_n(x,y) = 0$.
\end{itemize}
\end{proposition}

\bigskip

Define $$ \Delta (-J,z) = z^{-J(0)}exp \left ( \sum _{k=1} ^{\infty} (-1)^{k+1}\frac{-J(k)}{kz^k}\right ), $$ and let $$ \sum _{n \in \mathbb{Z}} \psi (a_n)z^{-n-1} = Y (\Delta (-J,z)a, z), $$ for $a \in \mathcal{W}^k$.

The operator $\Delta (h,z )$ associates to every $V$-module $M$ a new structure of an irreducible $V$-module. Let us denote this new module (obtained using the mapping $a_n \mapsto \psi (a_n)$) with $\psi (M)$.  As  the $\Delta$-operator acts bijectively on the set of irreducible modules, there exists an inverse  $\psi^{-1}(M)$.



\medskip

From the definition of $\Delta (-J,z)$ we have that
$$ \psi (J(n)) = J(n)- \frac{2k+3}{3} \delta _{n,0} , \quad \psi ( L(n)) = L(n) - J(n) + \frac{2k+3}{3} \delta _{n,0}, $$
$$ \psi (G^+(n)) = G^+(n-1), \quad \psi (G^-(n)) = G^-(n+1). $$


\begin{lemma} [\cite{A1}, Proposition 2.3.] \label{L_new}
	Let $ \dim L(x,y)_{\mbox{top}} = i$. Then $$ \widehat{ v(x,y)}  = G^+(0) ^{i-1} v(x,y)   $$ is a highest weight vector for $\psi (L(x,y))$, with highest weight $$ (x+i-1 -\frac{2k+3}{3}, y-x-i+1 +\frac{2k+3}{3}).$$ Specifically, it holds that $$ \psi(L(x,y)) \cong L(x+i-1 -\frac{2k+3}{3}, y-x-i+1 +\frac{2k+3}{3}). $$
\end{lemma}

\section{Classification of  irreducible modules for $k=-5/3$}

In this section we classify irreducible $\mathcal{W}_k$--modules for $k=-5/3$. Here we rely on results from Section 4, which connected the Zhu algebra associated to the Bershadsky-Polyakov algebra $\mathcal{W}^k$ to the Smith algebra $R(g)$. Another important ingredient in the classification is  the singular vector of conformal weight 4 in $\mathcal{W}^k$, and properties of its projection in the Zhu algebra.

\subsection{Singular vector and  relations in Zhu's algebra}
\bigskip

First we calculate an explicit formula for singular vector of conformal weight 4 and its projection in the Zhu algebra $A_{\omega}(\mathcal{W}^k)$.

\begin{lemma} \label{singW_4}
Vertex algebra $\mathcal{W}^k$ for $k=-\frac{5}{3}$ contains a unique (up to scalar factor) singular vector of conformal weight 4, and it is given by
\begin{align*}
	\Omega_4 &= -\frac{62}{9}L_{-2}^2\mathbbm{1} + \frac{14}{3}L_{-4}\mathbbm{1} - 18J_{-1}^4\mathbbm{1}  + 54J_{-2}J_{-1}^2\mathbbm{1} -130J_{-3}J_{-1}\mathbbm{1}+  \\& +\frac{33}{2}J_{-2}^2\mathbbm{1} + 13J_{-4}\mathbbm{1} -12L_{-3}J_{-1}\mathbbm{1}+ 46L_{-2}J_{-1}^2\mathbbm{1}  - G^{+}_{-2}G^{-}_{-1}\mathbbm{1} +  \\& + G^{+}_{-1}G^{-}_{-2}\mathbbm{1} -18J_{-1}G^{+}_{-1}G^{-}_{-1}\mathbbm{1}. & 
\end{align*}
\end{lemma}

Proof of this lemma is obtained by direct calculation and is therefore omitted.

\begin{proposition} \label{W_4}
	Let $\Omega_4$ be the singular vector of conformal weight 4 in $\mathcal{W}^k$ and let $k=-5/3$. Projection of  $\Omega_4$ in the Zhu algebra $A_{\omega}(\mathcal{W}^k)$ is given by the polynomial $U(x,y) \in {\Bbb C}[x,y]$,
	$$ U(x,y) = -18x^4+46x^2y-\frac{1}{2}x^2 -\frac{62}{9}y^2 -\frac{10}{9}y. $$
\end{proposition}
	
\medskip


Now we choose a new Virasoro field $$  L (z) = \sum _{n \in \mathbb{Z}}  L(n) z^{-n-2} := T(z) + \frac{1}{2} D  J(z),$$ so that the corresponding Zhu algebra $A_{\overline \omega}(\mathcal{W}^k)$ is a quotient of the Smith algebra $R(g)$ for $g(x,y) = -(3x^2 - (2k+3)x - (k+3)y)$ (cf. Proposition \ref{smith-zhu}). 

\medskip

From Proposition \ref{W_4} we obtain the following criterion:
\begin{lemma}\label{W4kriterij}
Let $L(x,y)$ be an irreducible $\mathcal{W}_k$--module. Then $(x,y) \in \mathbb{C}^2$ satisfy $$ U(x, y+ x/2)=0. $$
\end{lemma}

\medskip

The following relation (derived from the formula for the singular vector $\Omega_4$, see Appendix B) is essential in the classification of irreducible $\mathcal{W}_k$--modules. 

\begin {proposition} \label{zhu_k_5_3}
In the Zhu algebra $A_{\overline \omega}(\mathcal{W}_k)$ it holds that
\begin{equation} \label{eq:zhu_k_5_3}
  [G^+]^2([\overline \omega] + \frac{1}{9}) = 0.    
\end{equation}

\end{proposition}


\bigskip

 Define
$$ \mathcal S_k = \{  (-\frac{1}{9}, 0),  (0,0), (\frac{1}{3}, \frac{1}{3}), (-\frac{1}{3}, \frac{2}{3}), (-\frac{4}{9}, \frac{1}{3}), (-\frac{7}{9}, \frac{2}{3}) \}. $$ 
$$\widetilde{\mathcal  S_k} = \{ (\frac{1}{9}, - \frac{1}{9}), (\frac{4}{9}, - \frac{1}{9}), (\frac{7}{9}, - \frac{1}{9}) \}. $$

\begin{proposition}\label{klas_5_3}
Let $k =-5/3$.    
Assume that  $L(x,y)$ is an irreducible $\mathcal{W}_k$--module. Then  we have
\item[(i)] $(x,y) \in \mathcal S_k \cup  \widetilde{\mathcal  S_k}  $. 
\item[(ii)]  If $L(x,y)_{top}$ is finite-dimensional, then $(x,y) \in  \mathcal S_k$.
\end{proposition}

\begin{proof}

Let $M = \bigoplus_{n=0}^{\infty} M(n)$ be a $ \mathbb{Z}_{\geq 0}$-graded irreducible module for $\mathcal{W}_k $. From Zhu theory it follows that $M(0)$ is a module for the Zhu algebra $A_{\overline \omega}(\mathcal{W}^k)$, and that $$ [G^+]^2([\overline \omega] + 1/9) = 0$$ on $M(0)$. 

Let $v(x,y)$ be a highest weight vector with highest weight $(x,y)$ for $M(0)$. Assume that $M(0)$ is finite dimensional. From the above relation it follows that we have two cases: either $(G^+ (0))^2$ acts trivially on $M(0)$, or  $y=-1/9$.

 \medskip
 
\item[(1)] $(G^+ (0))^2 = 0$ on $M(0)$ 

\smallskip
	
	Since it holds that $(G^+(0))^2 v(x,y) = 0$ on $M(0)$ and $G^-(0)  (G^+(0))^i  v(x,y) = i h_i(x,y)(G^+(0))^{i-1}v(x,y)$, it follows that $h_2 (x,y) = 0$ or $G^+(0) v(x,y) = 0$ (hence $h_1(x,y) = 0$) $\Rightarrow$ $M(0)$ is either a 1-dimensional or 2-dimensional module for the Zhu algebra. Now we consider the module $\widehat M = \Psi (M)$. $\widehat M $ is a module with highest weight $(\hat x, \hat y)$, where
	$$ \hat x = x + i -1 + \frac{1}{9}, \quad \hat y = y - x - i + 1 - \frac{1}{9}.  $$
	
	Assume first that  $G^+(0) v(x,y) = 0$. Then it holds that
	$ \hat x = x + \frac{1}{9}, \quad \hat y = y - x - \frac{1}{9}.  $
	Notice that $\widehat y  =  - \frac{1}{9}$ if and only if $y = x$.
	If $y \ne x$, on the top level of $\widehat M$ again the relation $G^+ (0) =  0$ holds. It follows that the top level of $\widehat M(0)$ is $1$--dimensional and hence  $h_1(\hat x, \hat y) = 0$. Now 
	$$h_1(x,y)= 0  = h_1(\hat x, \hat y), $$  
	gives us a system of equations
	\begin{align*} 
	\begin{split}
	-3x^2 - \frac{1}{3}x + \frac{4}{3}y & =0 \\
	-3x^2 - \frac{7}{3}x + \frac{4}{3}y -\frac{2}{9}& =0 \\
	&
	\end{split}					
	\end{align*}
	which has a unique solution  $(x,y) = (-\frac{1}{9}, 0)$. \\
	If $x = y$, the equation $h_1(x,x) = 0$ has solutions $ x=0$ and $x =\frac{1}{3}$.
	
	\medskip
	Let us now consider the case $h_2(x,y) = 0$. It holds that  $ \hat x = x + \frac{10}{9}$, $\hat y = y - x - \frac{10}{9} $, and furthermore $\widehat y  =  - \frac{1}{9}$ if and only if $y = x+1$. 
	
	\smallskip
	
	(a) Suppose first that $y \neq x+1$. Then we again have the relation  $(G^+ (0))^2 =  0$ on the top level of $\widehat M$. Hence the top level of $\widehat M(0)$ is either $1$--dimensional or $2$--dimensional and it holds that  $h_1(\hat x, \hat y) = 0$ or $h_2(\hat x, \hat y) = 0$. \\
	
	If $M(0)$ is a 1-dimensional module for the Zhu algebra, it holds that
	$$ h_1(x,y) =0 = h_2( x + \frac{1}{9},  y - x - \frac{1}{9}),  $$ and 
	\begin{align*} 
	\begin{split}
	-3x^2 - \frac{1}{3}x + \frac{4}{3}y & =0 \\
	-3x^2 - \frac{16}{3}x + \frac{4}{3}y -\frac{20}{9}& =0 \\
	&
	\end{split}					
	\end{align*}
	which has a unique solution  $(x,y) = (-\frac{4}{9}, \frac{1}{3})$. \\
	
	If $M(0)$ is a 2-dimensional module for the Zhu algebra, there are two possibilities: \\
	if $$h_2(x,y)= 0  = h_1(\hat x, \hat y), $$ we obtain a system of equations
	\begin{align*} 
	\begin{split}
	-3x^2 - \frac{10}{3}x + \frac{4}{3}y - \frac{5}{3} & =0 \\
	-3x^2 - \frac{25}{3}x + \frac{4}{3}y -\frac{50}{9}& =0 \\
	&
	\end{split}					
	\end{align*}
	which has a unique solution  $(x,y) = (-\frac{7}{9}, \frac{2}{3})$; \\
	if $$h_2(x,y)= 0  = h_2(\hat x, \hat y), $$ we obtain a system of equations
	\begin{align*} 
	\begin{split}
	-3x^2 - \frac{10}{3}x + \frac{4}{3}y - \frac{5}{3} & =0 \\
	-3x^2 - \frac{34}{3}x + \frac{4}{3}y -\frac{95}{9}& =0 \\
	&
	\end{split}					
	\end{align*}
	which has a unique solution $(x,y) = (-\frac{10}{9}, \frac{5}{4})$. 
	
	From  Lemma \ref{W4kriterij} it follows that if $L(x,y)$ is an irreducible $\mathcal W_k$--module, then $(x,y)$ must be a zero of the equation $$ U(x, y + x/2) =0. $$ Since $(-\frac{10}{9}, \frac{5}{4})$ is not a solution of $U(x, y + x/2) = 0$, then $L (-\frac{10}{9}, \frac{5}{4})$ cannot be a $\mathcal{W}_k$--module.
	
	\medskip
	
	(b) Suppose that $y = x+1$. Since we have that $h_2(x,y)=h_2(x,x+1) = 0$, by solving this equation we obtain a unique solution $(x,y) = (-\frac{1}{3},-\frac{1}{3})$.
	
	\bigskip
	
	\item[(2)] It remains to consider   $\mathbb{Z}_{\ge 0}$--graded modules  $M = L(x,y) =\bigoplus_{n=0}^{\infty} M(n)$ with $ M(0)= L(x,y)_{top}$  such that
	$$(*) \quad   L(0)  \equiv -\frac{1}{9} \mbox{Id}  \quad \mbox{on} \ M(0). $$
	
	As above, a necessary condition for $L(x,y)$ to be an irreducible $\mathcal W_k$--module is that $(x,y)$ is a zero of the equation $$ U(x, y + x/2) =0. $$
By solving this equation, we get $x \in \{ \frac{1}{9}, \frac{4}{9}, \frac{7}{9}, -\frac{1}{18} \}$.

If   $L(-1/18, -1/9)$ is  a $\mathcal W_k$--module, then by Lemma \ref{simetrija-1} so is $L(1/18, -1/6)$. From the relation $$ [G^+]^2([\overline \omega] + 1/9) = 0 $$ on $L(1/18, -1/6)_{top}$, we see that (as $y = -1/6 \neq -1/9$), if $L(1/18, -1/6)$ were a $\mathcal W_k$--module, then necessarily $(G^+ (0))^2 = 0$ on $L(1/18, -1/6)_{top}$. Hence $L(1/18, -1/6)_{top}$ would need to be  a 1- or 2-dimensional module for the Zhu algebra.

But, since $h_1(x,y) \neq 0$ and $h_2(x,y) \neq 0$ for $(x,y) = (1/18,-1/6)$, the top component of $L(1/18,-1/6)$ is not a 1- or 2-dimensional module for the Zhu algebra. Hence $L(-1/18, -1/9)$  is not a  $\mathcal W_k$--module.

This proves that $(x,y) \in \mathcal S_k \cup  \widetilde{\mathcal  S_k}  $, so assertion (1) holds.

From the first part of the proof we have that $L(x,y)_{top}$ is finite-dimensional for $(x,y) \in  \mathcal S_k$.  Assume that $\dim   L(x,y)_{top} = i$ for certain $ i \in {\Z}_{>0}$ and $ (x,y) \in   \widetilde{\mathcal  S_k}  $. By direct calculation we have
 \begin{align*}   h_i(1/9 , -1/9) &=  2/9 + i - i^2\ne 0,  \\    h_i(4/9 , -1/9) &=  1/9 - i^2\ne 0, \\   h_i(7/9 , -1/9) &=  -(2/9) - i - i^2 \ne 0.  \end{align*}
A contradiction. Therefore, $ L(x,y)_{top}$ is infinite-dimensional for $ (x,y) \in   \widetilde{\mathcal  S_k}  $. This proves assertion (2).
\end{proof}

\subsection{Embedding of $\mathcal W_k$  into the  Weyl vertex algebra} 

It remains to prove that modules from Proposition \ref{klas_5_3} are indeed $\mathcal{W}_k $--modules. 
For that purpose, we will show first that the vertex algebra $\mathcal{W}_k $ 
 can be embedded into the Weyl vertex algebra. 

\begin{proposition} \label{ulaganje}
	Let
	$$J = -\frac{1}{3}a_{-1}^{+}a_{-1}^{-}\mathbbm{1}, \; 
	\omega = \frac{1}{2}\left(a_{-2}^{-}a_{-1}^{+}-a_{-2}^{+}a_{-1}^{-}\right)\mathbbm{1}, $$  $$ G^+ = \frac{1}{3}\left(a_{-1}^{+}\right)^{3}\mathbbm{1},\;
	G^- = \frac{1}{9}\left(a_{-1}^{-}\right)^{3}\mathbbm{1},  $$
	where $ \{ a^{\pm} _n  : n \in {\Z} \} $ are generators of the Weyl vertex algebra $W$. The vertex subalgebra $ \widetilde {\mathcal W_k}$ of the Weyl vertex algebra $W$ generated by vectors  $J,  \omega, G^{\pm}$  is isomorphic to a certain quotient of $\mathcal{W}^k$.
\end{proposition}

\begin{proof}[Sketch of proof]
We claim that the above choice of generators defines a nontrivial homomorphism of vertex algebras $\Theta : \mathcal{W}^k  \rightarrow W$. This follows from the following lemma:
\begin{lemma} \label{lema G}
For $J, \omega, G^{+}, G^{-}$ as above it holds that
    \begin{align*} 
			\begin{split}
			G_{2}^{+}G^{-} & = \frac{2}{9}\mathbbm{1},\\
			G_{1}^{+}G^{-} & =  -2J,\\
			G_{0}^{+}G^{-} & =  3J_{-1}^{2}-DJ-\frac{4}{3}\omega,
			\end{split}					
	\end{align*}
\end{lemma} and the fact that $\omega = \frac{1}{2}\left(a_{-2}^{-}a_{-1}^{+}-a_{-2}^{+}a_{-1}^{-}\right)\mathbbm{1}$ is a conformal vector of central charge $-1$. Proof of Lemma \ref{lema G} is technical and is given in the Appendix.
\end{proof}
In the next proposition we will show that $ \widetilde {\mathcal W_k}$ is in fact isomorphic to the simple quotient $\mathcal{W}_k$.

\medskip


Let $g = e^{ \frac{2 \pi i }{3}   J_0 } $. Then $g$ is an automorphism of $W$ of order $3$ and it holds that
$$ W = W^{(0)}  + W^{(1)}  + W^{(-1)}, $$
where $$W^{(j)} = \{ v \in W \vert \ g v = e^{   \frac{-2 \pi i }{3}   j  }   v\} , \quad j=0, 1, 2. $$ Hence $W^{(0)}$ is a simple vertex algebra, and  $W^{(\pm 1) } $ are irreducible $W^{(0)}$-modules.


\begin{lemma} Let $\Omega_4$ the singular vector of conformal weight 4 in $\mathcal{W}^k$.  We have:
$$ \Theta (\Omega_4) = 0  \quad \mbox{in} \ W. $$
\end{lemma}

\begin{proof}
Let $$W_{4,0} = \{ v \in W \vert L_0 v = 4 v, \ J_0 v = 0 \}, \quad \mathcal W^k _{4,0} =  \{ v \in \mathcal W^k  \vert L_0 v = 4 v, \ J_0 v = 0 \}. $$
First notice that 
\begin{align*}
 W_{4,0} = \mbox{span}_{\C} & \{ (a_{-2}^+)^2(a_{-1}^-)^2,(a_{-2}^-)^2(a_{-1}^+)^2,a_{-3}^+a_{-2}^-,a_{-3}^-a_{-2}^+,    a_{-4}^+a_{-1}^-,a_{-4}^-a_{-1}^+, &\\ 
 & a_{-3}^+a_{-1}^+(a_{-1}^-)^2,  a_{-3}^-a_{-1}^-(a_{-1}^+)^2, a_{-2}^+(a_{-1}^-)^3(a_{-1}^+)^2,  a_{-2}^-(a_{-1}^+)^3(a_{-1}^-)^2, &\\ 
 &  a_{-1}^+a_{-1}^-a_{-2}^+a_{-2}^-,  (a_{-1}^+)^4(a_{-1}^-)^4\}, &
\end{align*}
hence $ \dim  W_{4,0}  =  12$. Similarly we see that $ \dim  \mathcal W^k _{4,0}  =  13$. We conclude that there is a nontrivial relationship between the generators of $\mathcal{W}^k$ and hence $\text{Ker } \Theta \neq 0$. 

Since the singular vector $\Omega_4$ of conformal weight $4$ is unique (cf. Lemma \ref{singW_4}), it must hold that $\Omega_4 \in \text{Ker }\Theta $.
\end{proof}

\begin{proposition} \label{orbifold}
    Let  $W$  be the Weyl vertex algebra, $g$ as above. Then it holds that:
	\begin{itemize}
	\item[(1)]  $\mathcal{W}_k= W^{(0)}$.
	\item[(2)] $W^{(\pm 1)} $ are irreducible $\mathcal{W}_k$--modules of highest weight  $( \frac{1}{3}, \frac{1}{3})$, $( -\frac{1}{3}, \frac{2}{3})$ respectively.
	\end{itemize}	
\end{proposition} 
 
 \begin{proof} 
	First notice that $a^{\pm} \in W^{(\pm 1)}$ are highest weight vectors, with highest weights $( \frac{1}{3}, \frac{1}{3})$, $( -\frac{1}{3}, \frac{2}{3})$. If we show that $\widetilde{\mathcal W_k } =  \mathcal{W}_k = W^0$, then it will follow that $W^{(\pm 1)} $ are irreducible $\mathcal{W}_k$-modules of highest weight  $( \frac{1}{3}, \frac{1}{3})$, $( -\frac{1}{3}, \frac{2}{3})$.
	
	Let us prove (1). Assume that $\widetilde{\mathcal W_k}  \ne  W^{(0)}$. Then $W^{(0)}$ is a module for $\widetilde{\mathcal W_k }$. $ W^{(0)}$ contains an irreducible $\mathcal{W}^k$--subquotient $L(x,y) $ of highest weight $(x,y)$, where $x=m  \in {\Bbb Z}$, and $y = \frac{n}{2}  \in \frac{1}{2} {\Bbb Z}$. From Proposition \ref{klas_5_3} we know that if $L(x,y)$ is an irreducible $\mathcal{W}_k$--module with finite dimensional weight subspaces for $L(0)$, then $(x,y) \in \mathcal S_k $.
	But since potential highest weights also need to be of the form  $(x,y)= (m,\frac{n}{2} )$, for $m,n \in \mathbb{Z}$, we conclude that $(m,n) = (0,0)$ is the only such weight.
	Since $\widetilde{\mathcal W_k } = W^{(0)}$, $\widetilde{\mathcal W_k } $ is simple, hence  $\widetilde{\mathcal W_k } =  \mathcal{W}_k$. This completes the proof of the claim. 
\end{proof}

\subsection{Classification of irreducible modules in the category $\mathcal O$}

Now we are ready to complete the proof of classification of  ordinary $\mathcal W_k$--modules  and $\mathcal W_k$--modules in the category $\mathcal O$ for $k=-5/3$.
 So we need to determine which irreducible highest weight modules $L(x,y)$ are modules for $\mathcal{W}_k$.

\begin{theorem} \label{class-o-1}
Let $k =-5/3$. 
The set   $$  \{ L(x,y)  \ \vert  \ (x,y) \in \mathcal 
S_k \cup \widetilde{\mathcal S_k}\}, $$ gives a complete list of irreducible $\mathcal W_k$--modules from the category $\mathcal O$.
\end{theorem}

\begin{proof}
We have proved in Proposition \ref{klas_5_3}    that if $L(x,y)$ is a $\mathcal W_k$--module, then $(x,y) \in  \mathcal  S_k \cup \widetilde{\mathcal S_k}$. So it remains to see that modules parametrized by   $\mathcal  S_k \cup \widetilde{\mathcal S_k}$ are indeed $\mathcal W_k$--modules.

From Proposition  \ref{orbifold} it follows that the Weyl vertex algebra $W$ is a direct sum of three irreducible $\mathcal{W}_k$--modules, with the following highest weights:

\begin{itemize}
	\item $W^{(0)} $ has a highest weight vector $\mathbbm{1}$, with the highest weight $(0,0)$
	\item $W^{(1)} $ has a highest weight vector $a^+ _{-1} \mathbbm{1}$, with the highest weight $(1/3, 1/3)$
	\item $W^{(-1)} $ has a highest weight vector $a^- _{-1} \mathbbm{1}$, with the highest weight $(-1/3, 2/3)$.
\end{itemize}

Hence $L(0,0)$, $L(1/3, 1/3)$ and $L(-1/3, 2/3)$ are irreducible $\mathcal{W}_k$--modules. Claim  now follows from the formulas
\begin{align*}
\Psi ^{-1} (L(0,0)) &= L(-\frac{1}{9}, 0),  &\\
\Psi^{-1}   (L(-\frac{1}{3}, \frac{2}{3}))   &= L (-\frac{4}{9}, \frac{1}{3}), &\\
\Psi^{-1}  (L (\frac{1}{3}, \frac{1}{3}) ) &= L (-\frac{7}{9}, \frac{2}{3}), & \\
\Psi  (L(0,0)) &= L(\frac{1}{9}, -\frac{1}{9}),  &\\
\Psi   (L(-\frac{1}{3}, \frac{2}{3}))   &= L (\frac{7}{9}, -\frac{1}{9}), &\\
\Psi  (L (\frac{1}{3}, \frac{1}{3}) ) &= L (\frac{4}{9}, -\frac{1}{9}), & 
\end{align*}
and the fact that for every irreducible $\mathcal{W}_k$--module $M$,  $\Psi(M)$ and $\Psi^{-1} (M)$ are again irreducible $\mathcal{W}_k$--modules.
\end{proof}

 \begin{remark}
Note that the Weyl vertex algebra is non-rational and it contains infinitely many weight-modules (cf. \cite{AP-2019}) and modules of Whittaker type (cf. \cite{ALPY}). Since $\mathcal W_k$ is an orbifold of the Weyl vertex algebras, Whittaker and weight modules are also $\mathcal W_k$--modules. Moreover, the Weyl vertex algebra also contains logarithmic modules (cf. \cite{CR}, \cite{LMRS}), which by restriction, gives logarithmic $\mathcal W_k$--modules.
\end{remark}

\section{Classification of  irreducible modules for $k=-9/4$}


In this section we classify irreducible $\mathcal{W}_k$--modules for $k=-9/4$, using methods that are analogous to the ones used in the previous section to classify $\mathcal{W}_k$--modules for $k=-5/3$. Realization of these modules relies on a construction of the Bershadsky-Polyakov algebra $\mathcal{W}_k$ as a part of the series of $\mathcal B_p$-algebras (see Section 7).

\subsection{Singular vector $\Omega_3$ and the relation in the Zhu algebra}

As in the previous section, the starting point in the classification of irreducible $\mathcal{W}_k$--modules is the formula for the singular vector $\Omega_3$ of conformal weight 3, which is obtained by direct calculation.

\begin{lemma}  \label{sing_k_9_4}
Vertex algebra $\mathcal{W}^k$ for $k=-\frac{9}{4}$ contains a unique (up to scalar factor) singular vector of conformal weight 3, and it is given by
$$\Omega_3 = \frac{3}{8}L_{-3}\mathbbm{1} +  J_{-1}^3\mathbbm{1} -3J_{-2}J_{-1}\mathbbm{1} + \frac{11}{4}J_{-3}\mathbbm{1}  -\frac{3}{2}L_{-2}J_{-1}\mathbbm{1} + G^{+}_{-1}G^{-}_{-1}\mathbbm{1}.$$
\end{lemma}

\begin{proposition} 
		Let $\Omega_3$ be the singular vector of conformal weight 3 in $\mathcal{W}^k$ and let $k=-9/4$. Projection of  $\Omega_3$ in the Zhu algebra $A_{\omega}(\mathcal{W}^k)$ is given by the polynomial $V(x,y) \in {\Bbb C}[x,y]$,
	$$  V(x,y) = x^3 - \frac{3}{2}xy - \frac{5}{8} x = x ( x^2 - \frac{3}{2} y - \frac{5}{8}) . $$
\end{proposition}

Let $\overline{\omega}= \omega + \frac{1}{2}DJ$ be the new Virasoro vector. Similarly to the case $k=-5/3$, we derive the following key relation in $A_{\overline \omega}(\mathcal{W}_k)$ from the formula for the singular vector $\Omega_3$. 

\begin {proposition} \label{Gw}
In the Zhu algebra $A_{\overline \omega}(\mathcal{W}_k)$ it holds that
\begin{equation} \label{eq:zhu_k_9_4}
 [G^+]([\overline \omega] + \frac{1}{2}) = 0. 
\end{equation}
 \end{proposition}

Proof of relation (\ref{eq:zhu_k_9_4}) is completely analogous to the case $k=-5/3$ (see Appendix B).

\subsection{Classification of irreducible  modules in the category $\mathcal O$}

\begin{proposition} \label{klas_9_4}
Let $k =-9/4$.  Define
$$ \mathcal S_k = \{  (-\frac{1}{2}, 0),  (0,0), (-\frac{1}{4}, -\frac{1}{4}) \},  \quad   \widetilde { \mathcal S_k} = \{  (0, -\frac{1}{2}),  (\frac{1}{4},-\frac{1}{2}), (\frac{1}{2}, -\frac{1}{2}) \}.$$
\item[(i)]For every $(x,y) \in \mathcal S_k  \cup \widetilde{\mathcal S_k}$, $L(x,y)$ is a $\mathcal{W}_k$--module.
\item[(ii)] Assume that  $L(x,y)$ is an irreducible $\mathcal{W}_k$--module. Then $(x,y) \in \mathcal S_k  \cup \widetilde{  \mathcal S_k} $. 
\item[(iii)] Assume that  $L(x,y)$ is an irreducible $\mathcal{W}_k$--module with finite dimensional weight subspaces for $L(0)$. Then $(x,y) \in \mathcal S_k $. 

\smallskip

Therefore, the set $\{ L(x,y) \vert \ (x,y) \in   \mathcal S_k  \cup \widetilde{\mathcal S_k} \}$ gives all irreducible $\mathcal W_k$--modules in the category $\mathcal O$.
\end{proposition}

\begin{proof} 
Let us first prove (i). Here we will use the fact that if $M$ is an irreducible $\mathcal{W}_k$--module, then $\Psi (M)$ and $\Psi^{-1} (M)$ are again $\mathcal{W}_k$--modules. 

Since $\mathcal{W}_k = L(0,0)$, we have that $L(0,0)$ is a $\mathcal{W}_k$--module. Furthermore, as it holds that
$\Psi (L(-1/2, 0)) = L(0,0)$,  $L(-1/2,0) = \Psi^{-1} (L(0,0))$ is a module for the vertex algebra $\mathcal{W}_k$.

Using Lemma  \ref{lema_1/4_a} we have that $L(x,y)$ are $\mathcal W_k$--modules for $(x,y) \in   \widetilde { \mathcal S_k} $.  Since $L(-1/4, -1/4) =\Psi^{-1} (L( 1/4, -1/2))$, $L(-1/4, -1/4)$ is also a module for $\mathcal{W}_k$.
So, (i) holds.

\vskip 5mm

Let us prove (ii) and (iii).

Let $M = L(x,y) =  \bigoplus_{n=0}^{\infty} M(n)$ be a $ \mathbb{Z}_{\geq 0}$-graded irreducible module for $\mathcal{W}_k$. From Zhu theory it follows that $M(0) = L(x,y)_{top}$ is a module for the Zhu algebra $A_{\overline \omega} (\mathcal W_k)$, and that  $$[G^+]([\overline \omega] + 1/2) = 0$$ on $M(0)$. 

Let $v(x,y)$ be a highest weight vector with highest weight $(x,y)$ for $L(x,y)$. Assume that $M(0)$ is finite dimensional. From the above relation it follows that we have two cases: either $G^+ (0)$ acts trivially on $M(0)$, or  $y=-1/2$.

\medskip
 
\item[(1)] $G^+ (0) = 0$ on $M(0)$

\smallskip 
	
	Since it holds that $G^+(0) v(x,y) = 0$ on $M(0)$ and $G^-(0)  G^+(0) v(x,y)= g(x,y)v(x,y)$, it follows that $g(x,y) = h_1 (x,y) = 0$ $\Rightarrow$ $M(0)$ is a $1$-dimensional module for the Zhu algebra. Now we consider the module $\widehat M = \Psi (M)$. $\widehat M $ is a module with highest weight $(\hat x, \hat y)$, where
	$$ \hat x = x + \frac{1}{2}, \quad \hat y = y - x - \frac{1}{2}.  $$
	Notice that $\widehat y  =  - \frac{1}{2}$ if and only if $y = x$. If $y \ne x$ on the top level of $\widehat M$ again the relation $G^+ (0) =  0$ holds. It follows that the top level of $\widehat M(0)$ is $1$--dimensional and hence  $h_1(\hat x, \hat y) = 0$. Now 
	$$h_1(x,y)= 0  = h_1(\hat x, \hat y), $$  
	gives us a system of equations
	\begin{align*} 
	\begin{split}
	-3x^2 - \frac{3}{2}x + \frac{3}{4}y & =0 \\
	-3x^2 - \frac{21}{4}x + \frac{3}{4}y -\frac{15}{8}& =0 \\
	&
	\end{split}					
	\end{align*}
	which has a unique solution $(x,y) = (-\frac{1}{2}, 0)$.
	\medskip
	If $x = y$, the equation $h_1(x,x) = 0$ has solutions $ x=0$ and $x =-1/4$.
	
	\medskip
	
	(2)  Assume now that  $M = \bigoplus_{n=0}^{\infty} M(n)$   such that
	$$(*) \quad  L(0)  \equiv -\frac{1}{2} \mbox{Id}  \quad \mbox{on} \ M(0). $$
	
	Denote $M= L(x,y)$ and let  $y =-1/2$. Solving the equation $V(x, y+ x/2) = 0$, we get that $x= r /4$, $r \in \{ 0, 1,2\}$ and therefore $(x,y) \in \widetilde{\mathcal S_k}$.
	This proves assertion (ii). By (i) we have that $L(x,y)_{top}$ is finite-dimensional for $(x,y) \in \mathcal S_k$.  One can easily see that $h_i( r /4, - \frac{1}{2}) \ne 0 $ for each $i \in {\Z}_{>0}$, $ r=0, 1, 2$, since
	\begin{align*}  
	h_i(0 , -1/2) &=  -1/8 (2i-1)(4i-1)\ne 0,  \\ 
	h_i(1/2 , -1/2) &=  -1/8 (2i+1)(4i+1)\ne 0, \\   
	h_i(1/4 , -1/2) &=  -1/16(4i-1)(4i+1) \ne 0,  
	\end{align*} which implies that  $L(x,y)_{top}$ is infinite-dimensional. This proves assertion (iii).

The proof follows.
\end{proof}

\section{Modules for $\mathcal W_{-9/4}$  outside of the  category $\mathcal{O}$} 
\
Bershadsky-Polyakov vertex algebra $\mathcal{W}_k$ is a part of a series of vertex algebras which can be realized using vertex algebras from logarithmic conformal field theory, the so-called $\mathcal B_p$--algebras. For $p=3$, vertex algebra $\mathcal B_3$ is realized in the paper by D. Adamovi\' c  \cite{A-2005} as the affine vertex algebra associated to $\widehat{sl_2}$ at level $-4/3$. $\mathcal B_p$-algebras for $p \ge 4$ are defined in the paper by T. Creutzig, D. Ridout and S. Wood \cite{CRW}, where they conjectured that those algebras can be realized using quantum hamiltonian reduction. This statement hasn't been proven in full generality, but in the case $p=4$ it turns out that $\mathcal B_4$ coincides with the Bershadsky-Polyakov algebra $\mathcal{W}_k$ at level $k=-9/4$.

\subsection{$\mathcal B_4$--algebra realization}



 It is conjectured that  $\mathcal B_p$ is isomorphic to the affine $W$--algebra $W_{p-1} ^{(2)}$ introduced by Feigin and Semikhatov.  The authors prove this conjecture for $p \le 5$. For $p=4$, this gives the Bershadsky-Polyakov vertex algebra.
 
 In this section  we shall recall the definition of  $\mathcal B_p$ in the case $p=4$, and present an alternative proof of simplicity which uses the representation theory of  $\mathcal W_{k}$ for $k=-9/4$.

We start by describing the construction of the doublet vertex algebra $\mathcal A(4)$, which is a special case of the series of vertex algebras $\mathcal A(p)$ from the paper of  D. Adamovi\' c and A. Milas \cite{AdM-doublet}.

Let $L$ be an even lattice such that
$$ L = \mathbb{Z} \gamma + \mathbb{Z}\delta \, \quad \langle \gamma, \gamma  \rangle = - \langle \delta, \delta \rangle = 2, \ \langle \gamma, \delta \rangle = 0. $$ 

Let $\mathcal A(4)$ be the vertex algebra generated by $$ a^- = e^{-\gamma}, a^+ = Q a^-, \omega_{\mathcal A (4)  }= \frac{1}{4} \gamma (-1) ^2 + \frac{3}{4} \gamma(-2),  $$  
where $Q = e^{2 \gamma}_0$ is a screening operator.

We shall need the following formulas
\begin{align*}  
&a^+ _4 a^- = - a^- _ 4 a^ +  = - 20\mathbbm{1}, \\
&a^+ _3 a^- = a^- _ 3 a^ +  =0,  \\
&a^+ _2  a^- = - a^- _ 2 a^ + = 2 \gamma(-1) ^2 + 6 \gamma(-2) =  8 \omega_{\mathcal A(4)}. \\
&a^+ _1  a^- = A   D \omega_{\mathcal A(4)}. \quad (A = 4).
\end{align*}
 
Then the vertex algebra $\mathcal B_4$ (cf. \cite{CRW}) is defined as a subalgebra of $\mathcal A(4) \otimes V_{\Z \delta} \subset V_L$, generated by
\begin{align*}
\tau^- &= - \frac{\sqrt{6} }{8}a^- \otimes e^{-\delta}=  - \frac{\sqrt{6} }{8}  e^{-\gamma-\delta}, &\\  
\tau^+ &=  \frac{\sqrt{6} }{8} a^+ \otimes e^{\delta} =  \frac{\sqrt{6} }{8}  \frac{1}{3} ( 4  \gamma(-1) ^3 + 6 \gamma(-1) \gamma(-2) + 2 \gamma(-3) ) e^{\gamma+ \delta}, &\\ 
j &= -\frac{1}{2} \delta(-1), &\\ 
\omega &= \omega_{\mathcal A (4)} - \frac{1}{4} \delta(-1) ^2. & 
\end{align*}

Direct computation shows that
\begin{align*}
\tau^+ _2 \tau^- &= \frac{15}{8}\mathbbm{1}  = (k+1) (2k +3)\mathbbm{1},  &\\
\tau^+ _1 \tau^- &=  \frac{15}{8} \delta(-1) = 3 (k+1)  j, &\\
\tau^+ _0 \tau^- &= -\frac{6}{64} (8 \omega_{\mathcal A(4)} - 10 \delta(-1) ^2 - 10 \delta(-2))  &\\
&= -(k+3) \omega + 3 J_{-1} ^2 + \frac{3 (k+1)}{2} J_{-2}. & 
\end{align*}


  \begin{proposition} For $k=-9/4$ it holds that
        $$\mathcal{W}_k \cong  \mathcal B_4. $$
    \end{proposition}

\begin{proof} The above computation shows that $G^{\pm} = \tau^{\pm}, j, \omega$ generate a subalgebra of $\mathcal A(4) \otimes V_{\Z \delta}$  isomorphic to a certain quotient of the universal Bershadsky-Polyakov algebra $\mathcal{W}^k$. Hence $\mathcal B_4$ is a certain quotient of  $\mathcal{W}^k$.

In the realization, we have
\begin{align*}   
 \tau^+ _{-1} \tau^-   =  &    -\frac{3}{32} ( a^+ _{1} a^-    +   a^+ _{1} a^-  \delta(-1) -\frac{20}{6} (\delta(-1) ^3 + 3 \delta(-2) \delta(-1)  + 2 \delta(-3) ) )  & \\
  =&  -\frac{3}{32}  (  A D \omega_{\mathcal A(4)}   +   8  \omega_{\mathcal A(4)}  \delta(-1)-\frac{20}{6} (\delta(-1) ^3 + 3 \delta(-2) \delta(-1)  + 2 \delta(-3) ) ) & \\
  =&  -\frac{3}{32}  ( A L_{-3} + \frac{A}{2} \delta(-2) \delta(-1)  + 8 L_{-2} \delta( -1)  + 2 \delta(-1) ^3  - 8 \delta(-3) - &\\
   &-\frac{20}{6} (\delta(-1) ^3 + 3 \delta(-2) \delta(-1)  + 2 \delta(-3) )  ) & \\
  =&  -\frac{3}{32}  ( A L_{-3} + \frac{A}{2} \delta(-2) \delta(-1)  + 8 L_{-2} \delta( -1) - \frac{8}{6}  \delta(-1) ^3  - 8 \delta(-3)  - &\\
  & - \frac{20}{6} (  3 \delta(-2) \delta(-1)  + 2 \delta(-3) ) ) & \\
  =&  -\frac{3}{32}  (  4 L_{-3}  - 8  \delta(-2) \delta(-1)  + 8 L_{-2} \delta( -1) - \frac{4}{3}  \delta(-1) ^3   -\frac{44}{3}  \delta(-3) ) & \\ 
  =& - \frac{3}{8} L_{-3}  + 3 J_{-2} J_{-1}  + \frac{3}{2} L_{-2} J_{-1}   - J_{-1}  ^3 - \frac{11}{4}  J_{-3},  & 
\end{align*} 
 which implies that the singular vector $\Omega_3$ (cf. Lemma \ref{sing_k_9_4}) vanishes in the realization, i.e.,  $ \Omega_3 = 0 $ in   $\mathcal B_4$.

 We can consider   $\mathcal B_4$ as a ($\Z_{\ge 0}$--graded) VOA with Virasoro vector $\overline \omega  = \omega + \frac{1}{2} D J $. Clearly,   for every irreducible subquotient $L(x,y)$  of  $\mathcal B_4$, we have $\dim L(x,y)_{top} < \infty$.

Assume that $\mathcal B_4 \ne \mathcal W_k$. Then by Proposition \ref{klas_9_4} (ii), (iii), $\mathcal B_4$ has an irreducible subquotient isomorphic to $L(x,y)$ for  $(x,y) \in \mathcal S_k \setminus \{ (0,0) \}  = \{ (-\frac{1}{2},0), (-\frac{1}{4}, -\frac{1}{4}) \}$. On the other hand, there are no vectors in $\mathcal B_4$ having such weight.
A contradiction. This  implies  that $\mathcal B_4$  is a simple vertex algebra and hence isomorphic to  $\mathcal{W}_k$.

\end{proof}  
    
\subsection{Construction of a family of weight modules outside of the  category $\mathcal{O}$} 
  
Let us choose a new Virasoro vector $$L= \omega - \frac{1}{4} \delta(-2). $$ Let $D= {\Z} (\gamma+ \delta)$. Then $\mathcal{W}_k$ is realized as a subalgebra of $$\Pi(0) =M(1) \otimes {\C}[D]. $$ Specifically, for every $s \in \mathbb{Z}$ and $r \in \mathbb{C}$ it holds that
$$\mathcal M _s (r) := \Pi(0). e^{s  \delta+ r (\gamma +\delta ) }$$   is an irreducible $\Pi(0)$--module (cf. \cite{BDT}, \cite{LW}).

It holds that
\begin{align*}
L(n) e^{s  \delta+ r (\gamma +\delta ) } &=  0 \quad (n \ge 1) &\\
L(0) e^{s  \delta+ r (\gamma +\delta ) } &= \frac{4 r^2 - 6 r - 4 (r+s) ^2 - 2(r+s)  }{4}  e^{s  \delta+ r (\gamma +\delta ) } &\\ 
&= \frac{-8 r - 8 r s - 4 s^2 - 2 s }{4} e^{s  \delta+ r (\gamma +\delta )} & \end{align*}

From these formulas it follows that $ \mathcal M _s (r)$ is $\mathbb{Z}_{\ge 0}$--graded if and only if $s=-1 $. Then for every $r$ it holds that
$$ L(0)  e^{-  \delta+ r (\gamma +\delta ) } =-\frac{1}{2} e^{-\delta+ r (\gamma +\delta ) }.$$
In this way we have constructed an infinite series of  $\mathbb{Z}_{\ge 0}$--graded $\mathcal{W}_k$--modules  with lowest conformal weight $-\frac{1}{2}$.

\begin{theorem} \label{M_r_modul}
    \begin{enumerate}
    \item For every $r \in \mathbb{C}$, $\mathcal M _{-1} (r)$ is a $\mathbb{Z}_{\ge 0}$--graded $\mathcal{W}_k$--module with lowest conformal weight $-1/2$:
	$$ \mathcal M _{-1} (r) = \bigoplus _{m =0} ^{\infty}  \mathcal M _{-1} (r)  (m),  \quad L(0) \vert  \mathcal M _{-1} (r)  (m) \equiv (-\frac{1}{2} + m) \mbox{Id}. $$
    \item Assume that  $ 4 r \notin \mathbb{Z}$ for some $r \in \mathbb{C}$.  Then: 
    \begin{itemize} 
      \item[(i)] $\mathcal M _{-1} (r)  (0)$ is an irreducible module for the Smith algebra $R(g)$, for
	$$ g (x,y) = - (3 x^2  - (2k+3) x -(k+3) y)$$
	where $y= -1/2$.
	  \item[(ii)] $\mathcal{W}_k$--module $\mathcal M _{-1} (r)$ is irreducible.
	\end{itemize}
    \end{enumerate}
\end{theorem}

\begin{proof} 

(1) First notice that
$ e^{-  \delta+ r (\gamma +\delta ) }$ is a vector with $(J(0), L(0))$--weight $$ (x,y ) = (r -1, -1/2).$$
It holds that
$$
G^- (0) e^{-  \delta+ r (\gamma +\delta ) } = \tau^- _1  e^{-  \delta+ r (\gamma +\delta ) }  =  -  \nu e^{-  \delta+ (r-1) (\gamma +\delta ) }  \quad (\nu = \frac{\sqrt{6}}{8}). 
$$
From the relation
$$e^{-  \delta+ r (\gamma +\delta ) }  _ 0 \tau^+  =   - \nu  {  4 r   \choose 3 } e^{-  \delta+ ( r+1) (\gamma +\delta ) } $$ we have that
$$G^+  (0) e^{-  \delta+ r (\gamma +\delta ) }  =  \tau^+ _0 e^{-  \delta+ r (\gamma +\delta ) }   =  \nu  { 4 r    \choose 3 }  e^{-  \delta+ ( r+1) (\gamma +\delta ) }. $$
Now
\begin{align*}
[G^+  (0), G^-(0)] e^{-  \delta+ r (\gamma +\delta ) }  &= \left( - \nu ^2   {  4 r - 4    \choose 3 }  +  \nu^2 { 4  r   \choose 3 }   \right) e^{-  \delta+ r (\gamma +\delta ) } &\\
&= (3r^2 - 9r/2 + 15/8)e^{-  \delta+ r (\gamma +\delta ) } &\\
&= - (3 x^2  - (2k+3) x -(k+3) y)e^{-  \delta+ r (\gamma +\delta ) } &\\
&=  g( x, y)  e^{-  \delta+ r (\gamma +\delta ) }. &
\end{align*}

(2) From the above formulas it follows that the  top component  of the module  $\mathcal M _{-1} (r)$ can be realized as $$ \mathcal M _{-1} (r) (0)  = \mbox{span}_{\mathbb{C}} \{ e^{-  \delta+ (m + r) (\gamma +\delta ) }  \ \vert \ m \in \mathbb{Z} \}$$ and that it is also a module for the Smith algebra $R(g)$, where $g(x,y) =- ( 3 x^2  - (2k +3) x - (k+3) y$.  The Smith algebra $R(g)$ acts with
\begin{align*}
	E  e^{-  \delta+ (m + r) (\gamma +\delta ) }  &= G^+ (0)  e^{-  \delta+ (m + r) (\gamma +\delta )  }=  \nu  { 4 (r+m)    \choose 3 }  e^{-  \delta+ (m+  r+1) (\gamma +\delta ) }  &\\
	F  e^{-  \delta+ (m + r) (\gamma +\delta ) }  &= G^- (0)  e^{-  \delta+ (m + r) (\gamma +\delta ) }=  - \nu     e^{-  \delta+ ( m+r- 1) (\gamma +\delta ) }  &\\
	X  e^{-  \delta+ (m + r) (\gamma +\delta ) } &=  J(0)    e^{-  \delta+ (m + r) (\gamma +\delta ) }  = (m + r -1)  e^{-  \delta+ (m + r) (\gamma +\delta ) }   &\\
	Y  e^{-  \delta+ (m + r) (\gamma +\delta ) } & = L(0)    e^{-  \delta+ (m + r) (\gamma +\delta ) }  = -\frac{1}{2} e^{-  \delta+ (m + r) (\gamma +\delta ) } &
\end{align*}
Now it follows that $\mathcal M_{-1} (r) _{top}$ is an irreducible $R(g)$--module if the condition $  { 4 (r+m)\choose 3 }  \ne 0 $ holds for every $m \in \mathbb{Z}$, which is satisfied if  $ 4 r \notin \mathbb{Z}$. This concludes the proof of (i).
	
Let us prove (ii). Assume that   $\mathcal M _{-1} (r)$  is reducible. Since $\mathcal M _{-1} (r) (0)$ is an irreducible $A(\mathcal W_k)$--module, we have that  $\mathcal M _{-1} (r)$ has     a ${\Z}_{\ge 0}$-graded subquotient 
$$ M= \bigoplus _{n \in {\Z}_{\ge 0}} M(n), \quad M(0)  \cap \mathcal M _{-1} (r) (0) = \{ 0 \} .$$
Moreover, the  top component  $M(0)$  is then a module for the Zhu algebra $A(\mathcal W_k)$.
From the relation $[G^{+}] ([\omega] + 1/2) =0$ in the Zhu algebra  it follows that $M(0)$  is  a $1$--dimensional $A(\mathcal W_k)$--module. Now Proposition \ref{klas_9_4}  implies that $M(0)$  has  to have lowest conformal weight $0$  or $-1/4$.  But this is impossible, since there is no vector of conformal weight $0$ or $-1/4$ in $\mathcal M _{-1} (r)$. This concludes the proof of the theorem. 

\end{proof}

As a consequence, we can construct $3$ irreducible $\mathcal{W}_k$--modules in the category $\mathcal O$:
\begin{lemma} \label{lema_1/4_a}
	$L(r/4,-1/2)$ is a $\mathcal{W}_k$--module, which is realized as a subquotient of the $\mathcal{W}_k$--module $\mathcal M_{-1} (r)$ for $r = 0,1,2 $.
\end{lemma}
\begin{proof}
	For $r = 0,1,2 $ it holds that
	\begin{align*}
	G^+  (0) e^{-  \delta+ r (\gamma +\delta ) }  =  \tau^+ _0 e^{-  \delta+ r (\gamma +\delta ) }   &=   \nu  { 4 r    \choose 3 }  e^{-  \delta+ ( r+1) (\gamma +\delta ) } &\\
	&=  \nu  { 1    \choose 3 }  e^{-  \delta+ ( r+1) (\gamma +\delta ) } &\\
	&= 0, &
	\end{align*} 
	hence $U:= \langle e^{-  \delta+ r (\gamma +\delta ) } \rangle $ is a proper submodule of $\mathcal M_{-1} (r)$.
	
	\smallskip
	
	Let $v(x,y) = e^{-  \delta+ (r+1)(\gamma +\delta ) }$. Since
	\begin{align*}
	J(0)v(x,y) &= J(0)e^{-  \delta+ (r+1 )(\gamma +\delta ) } = re^{-  \delta+ (r+1) (\gamma +\delta ) } = rv(x,y), &\\
	L(0)v(x,y) &= L(0)e^{-  \delta+ (r+1 )(\gamma +\delta ) } = -\frac{1}{2}e^{-  \delta+ (r+1 )(\gamma +\delta ) }= -\frac{1}{2}v(x,y), &\\
	G^{-}(0)v(x,y) &= G^{-}(0)e^{-  \delta+ (r+1 )(\gamma +\delta ) } = \nu e^{-  \delta+ r(\gamma +\delta ) } + U \in U,
	\end{align*}
	it follows that $ v(x,y)$ is a highest weight vector of weight
	$$ (x,y ) = (r , -1/2) = (r/4, -1/2)$$ in the quotient $\mathcal M_{-1} (r) / U$. Therefore,  $L(r/4,-1/2)$ is a $\mathcal{W}_k$--module.
\end{proof}

\section{Irreducible $\mathcal{W}_k$-modules for integer levels $k$}

In this section we study irreducible highest weight modules for the Bershadsky-Polyakov algebra $\mathcal{W}_k$ at integer levels $k$, $k \geq -1$. We generalize a construction of a family of singular vectors from the paper \cite{A1} and prove a necessary condition for highest weights of $\mathcal{W}_k$--modules, while finding a realization for these as $\mathcal{W}_k$--modules remains an open question. In this paper we obtain a classification for levels $k=-1, 0$ using results from papers by D. Adamovi\' c, V. G. Kac, P. M\" oseneder-Frajria, P. Papi, O. Per\v se \cite{AKMPP-2018}; and T. Arakawa, T. Creutzig, K. Kawasetsu, A. Linshaw \cite{ACKL}.

\subsection{Singular vectors and a necessary condition for $\mathcal{W}_k$--modules}

First we generalize a construction of a family of singular vectors by T. Arakawa in \cite{A1}, where he found a similar formula for singular vectors in $\mathcal W^k$ at levels $k= p/2-3$, $p \ge 3$, $p$ odd.

\begin{lemma} \label{integral-sing}
	Vectors $$G^+(-1)^{n}\mathbbm{1}, \; G^-(-2)^{n}\mathbbm{1}$$ are singular in $\mathcal{W}^k$ for $n = k+2$, where $k \in \mathbb{Z}$, $k \geq -1$.
\end{lemma}
\begin{proof}

Observe that  $$J(m)G^+(-1)^{n}\mathbbm{1} =  L(m)G^+(-1)^{n}\mathbbm{1} = J(m)G^-(-2)^{n}\mathbbm{1} =  L(m)G^-(-2)^{n}\mathbbm{1} = 0$$ for $m > 0$. 
Also, it is easy to see that $G^-(m)G^+(-1)^{n}\mathbbm{1} = 0$ for $m > 1$ and $G^+(m)G^-(-2)^{n}\mathbbm{1} = 0$ for $m > 2$.

\medskip

To prove that $G^-(-2)^{k+2}\mathbbm{1}$ is a singular vector in $\mathcal{W}^k $, we need to show that $G^+(1)G^-(-2)^{n}\mathbbm{1} = G^+(2)G^-(-2)^{n}\mathbbm{1} = 0$ for $n = k+2$. We claim that
	\begin{equation*}
	\begin{split}
	G^+(1)G^-(-2)^{n}\mathbbm{1}& = 3n(k-(n-2))J(-1)G^-(-2)^{n-1}\mathbbm{1} + \\ & n(n-1)(k-(n-2))G^-(-3)G^-(-2)^{n-2}\mathbbm{1}.
	\end{split}
	\end{equation*}
We will prove this by induction. Assume that
	\begin{equation*}
	\begin{split}
	G^+(1)G^-(-2)^{n-1}\mathbbm{1}& = 3(n-1)(k-(n-3))J(-1)G^-(-2)^{n-2}\mathbbm{1} + \\ & (n-1)(n-2)(k-(n-3))G^-(-3)G^-(-2)^{n-3}\mathbbm{1}.
	\end{split}
	\end{equation*}
It holds that
	\begin{flalign*}
	G^+(1)G^-(-2)^{n}\mathbbm{1} 
	&= (3J^2(-1) + 3(k+1)J(-1) - (k+3)L(-1))G^-(-2)^{n-1}\mathbbm{1} + &\\ & + G^-(-2)G^+(1)G^-(-2)^{n-1}\mathbbm{1}&\\
	&= (3kn-3n^2+6n)J(-1)G^-(-2)^{n-1}\mathbbm{1} + \\& +(kn^2-kn-n^3+3n^2-2n)G^-(-3)G^-(-2)^{n-2}\mathbbm{1} \\
	&= 3n(k-(n-2))J(-1)G^-(-2)^{n-1}\mathbbm{1} + \\& +n(n-1)(k-(n-2))G^-(-3)G^-(-2)^{n-2}\mathbbm{1}.
	\end{flalign*}
	
Next, we claim that
	$$  G^+(2)G^-(-2)^{n}\mathbbm{1} = 2n(k-(n-2))(k-(n-2)+n/2)G^-(-2)^{n-1}\mathbbm{1}. $$
Assume that
	$$  G^+(2)G^-(-2)^{n-1}\mathbbm{1} = 2(n-1)(k-(n-3))(k-(n-3)+(n-1)/2)G^-(-2)^{n-2}\mathbbm{1}. $$
	
We have
	\begin{flalign*}
	G^+(2)G^-(-2)^{n}\mathbbm{1} 
	&= (3J^2(0) + (4k+3)J(0) - (k+3)L(0) + \\& (k+1)(2k+3))G^-(-2)^{n-1}\mathbbm{1} +  G^-(-2)G^+(2)G^-(-2)^{n-1}\mathbbm{1}&\\
	&= 3(n-1)^2G^-(-2)^{n-1}\mathbbm{1} - (n-1)(4k+3)G^-(-2)^{n-1}\mathbbm{1} \\& -(2n-2)(k+3)G^-(-2)^{n-1}\mathbbm{1}+ (k+1)(2k+3)G^-(-2)^{n-1}\mathbbm{1}  + \\& +2(n-1)(k-(n-3))(k-(n-3)+(n-1)/2)G^-(-2)^{n-2}\mathbbm{1} &\\
	&= 2n(k-(n-2))(k-(n-2)+n/2)G^-(-2)^{n-1}\mathbbm{1}.
	\end{flalign*}
	
\medskip
	
Similarly, to show that $(G^+(-1))^{k+2}\mathbbm{1}$ is a singular vector in $\mathcal{W}^k$, we need to show that  $G^+(1)G^-(-2)^{n}\mathbbm{1} =0$ for $n = k+2$. We claim that
	$$  G^-(1)G^+(-1)^{n}\mathbbm{1} = -n(k-(n-2))(2k-(n-4))G^+(-1)^{n-1}\mathbbm{1}. $$
Again we proceed by induction. Assume that	
	$$  G^-(1)G^+(-1)^{n-1}\mathbbm{1} = -(n-1)(k-(n-3))(2k-(n-5))G^+(-1)^{n-2}\mathbbm{1}. $$
	
It holds that
	\begin{flalign*}
	G^-(1)G^+(-1)^{n}\mathbbm{1} 
	&= (-3J^2(0) + (5k+6)J(0) + (k+3)L(0) - \\& - (k+1)(2k+3))G^+(-1)^{n-1}\mathbbm{1} +  G^+(-1)G^-(1)G^+(-1)^{n-1}\mathbbm{1}&\\
	&= ( -3(n-1)^2 + (n-1)(5k+6) +(n-1)(k+3) -(k+1)(2k+3)-&\\ & -(n-1)(k-(n-3))(2k-(n-5)))G^+(-1)^{n-1}\mathbbm{1} &\\
	&= -n(k-(n-2))(2k-(n-4))G^+(-1)^{n-1}\mathbbm{1}.
	\end{flalign*}

\end{proof}

Define the set 
$$  \mathcal S_k= \left\{  (x,y)  \in {\C} ^ 2 \ \vert  h_i(x,y) = 0, \; 1 \leq i \leq k+2 \right\}. $$
\begin{proposition} \label{klas-1}
	Let $k \in \mathbb{Z}$, $k \geq -1$. Then we have:
	\begin{itemize}
	\item[(1)] The set of equivalency classes of irreducible ordinary $ \mathcal W_k$--modules   is contained in the set $$ \{ L(x,y) \ \vert \  (x,y) \in  \mathcal S_k \}. $$
	\item[(2)]  Every irreducible   $ \mathcal W_k$--module in the category $\mathcal O$ is an ordinary module.
	\end{itemize}
\end{proposition}

\begin{proof}
	Let $L(x,y)$ be an irreducible $ \mathcal{W}_k$--module with a highest weight vector $v(x,y)$. 
	Since $$(G^+(- 1))^{k+2}\mathbbm{1} = 0$$ in  $\mathcal{W}_k$, it follows that $G^+(z)^{k+2}v(x,y)=0$. Hence on the top level $L(x,y)_{top}$ the following relation holds:
	$$(G^+(0))^{k+2} v_{x,y}= 0.$$  
	
	
	It follows that the space $L(x,y)_{top}$ is $i$-dimensional for certain  $ 1 \le i \le k+2$. So $L(x,y)$ is an ordinary module and from Proposition \ref{l_top} we get that $h_{i} (x,y)  = 0$, for certain $ 1 \le i \le k+2$.
	%
	In this way we have shown that the set of equivalency classes of irreducible $ \mathcal{W}_k$--modules is contained in the set
	 $$    \{ L(x,y) \ \vert \  (x,y) \in  \mathcal S_k \}. $$	
\end{proof}

\begin{conj}
The set   $ \{ L(x,y) \ \vert  \ (x,y) \in  \mathcal S_k \} $	 is the set of all irreducible ordinary $\mathcal W_k$--modules.
\end{conj}

In what follows, we will prove this conjecture for $k=-1, 0$.

\subsection{The vertex algebra $\mathcal W_{-1}$}

In \cite{AKMPP-2018} it was shown that the  Bershadsky-Polyakov algebra $\mathcal{W}_k$ for $k=-1$ is isomorphic to the Heisenberg vertex algebra $M(1)$. We will use this fact to prove the following:

\begin{theorem} The set
	$$   \{ L(x,y) \ \vert \ (x,y) \in  \mathcal S_{-1} \} $$
	 is the set of all irreducible $\mathcal{W}_{-1}$--modules.
\end{theorem}

\begin{proof}
For $k =-1$ the generators $G^{\pm}$  of $\mathcal{W}^k$ belong to the maximal ideal, hence $G^{\pm} = 0$ in the simple quotient $\mathcal{W}_k$ and $ \mathcal{W}_{-1} \cong M(1). $

Also, for $k=-1$ there is a conformal embedding of the  Heisenberg vertex algebra into $\mathcal{W}_k$ (cf. \cite{AKMPP-2018}). The (original) Virasoro vector is given by   $\omega =\frac{3}{2} :J^2:$, 
and the shifted Virasoro is
$$  \overline \omega  = \omega + \frac{1}{2} DJ =  \frac{3}{2}J(-1) ^2 \mathbbm{1} + \frac{1}{2} J(-2)\mathbbm{1}. $$

Since all irreducible modules for the Heisenberg vertex algebra $M(1)$ are of the form $M(1, x)$ for some $x \in {\C}$, where $J(0)$ acts on $M(1,x)$ as $x \mbox{Id}$, then $L(0)$ acts on the highest weight vector of $M(1,x)$ as
$$ y = \frac{3}{2} x^2 - \frac{1}{2} x.$$

This shows that the highest weights of irreducible $\mathcal W_{-1} = M(1)$--modules coincide with the zeroes of the polynomial $h_1(x,y) = -3x^2 + (2k+3)x + (k+3)y$ for $k=-1$.
\end{proof}

\subsection{The vertex algebra $\mathcal W_0$ and its modules.}

Let $\mathcal A (1) $ be the vertex algebra of symplectic fermions generated by odd fields  $b$ and $c$, with the Virasoro vector  $\omega_{\mathcal A(1)}  = :b c: $ of central charge $c=-2$.

Let $F$ be the Clifford vertex algebra generated by odd fields $\Psi^{+}$ and $\Psi^{-}$ and let $M(1)$ be the Heisenberg vertex subalgebra of  $F$ generated by $\alpha:= :\Psi^+ \Psi^- :$. Vertex algebras  $M(1) $ and $F$ have a Virasoro vector $\omega_F = \frac{1}{2} :\alpha \alpha:$  of central charge $c=1$. 

We will need the following result from \cite{ACKL} :
\begin{theorem} \cite{ACKL}  \label{real-ferm} 
There exists a non-trivial homomorphism of vertex algebras
	\begin{align*}
	\Phi: \mathcal{W}_0 &\rightarrow  F\otimes \mathcal A(1),  &\\
	J&\mapsto   :\Psi^+ \Psi^- :  &\\
	T & \mapsto  \omega_{F} + \omega_{\mathcal A(1)} &\\
	G^+ & \mapsto  \sqrt{3} : \Psi^+  b: &\\
	G^- &\mapsto   \sqrt{3} : \Psi^-  c:. &
	\end{align*}
	\end{theorem}

\medskip

Let  $V_L =M(1) \otimes \mathbb{C} [ L] $ be the lattice vertex algebra associated with the lattice  $L= \mathbb{Z}\alpha_1 + \mathbb{Z}\alpha_2$, where
$$\langle \alpha_i, \alpha_j\rangle = \delta_{i,j}, \quad i,j =1,2.$$

We consider the subalgebra $V[D]  = M(1)\otimes  \mathbb{C}[D]$  of $V_L$, where $D= {\Z} (\alpha_1+ \alpha_2)$.

For every $x \in \mathbb{C}$, $i=0,1$, $$V[D- i \alpha_1 - x (\alpha_1-\alpha_2)] = V[D]. e^{- i\alpha_1 - x (\alpha_1-\alpha_2)}$$ is an irreducible $V[D]$--module.

\begin{theorem}\label{class-k0}
\begin{itemize}
    \item[(1)] The simple vertex algebra $\mathcal{W}_0$ can be realized as a vertex subalgebra of $V[D]$ generated by vectors 
    \begin{align*}
	J&\mapsto   \alpha_2(-1)   &\\
	L & \mapsto  \frac{1}{2} \left( \alpha_1 (-1) ^2 - \alpha_1(-2) + \alpha_2(-1) ^2 + \alpha_2(-2)  \right) &\\
	G^+ & \mapsto \sqrt{3} e^{\alpha_1 + \alpha_2} &\\
	G^- &\mapsto  -\sqrt{3} \alpha_1(-1) e^{-\alpha_1-\alpha_2}. & 
	\end{align*}
	\item[(2)]  $\mathcal{W}_0$  has two families of  irreducible  highest weight modules
	$$ \{ L(x, x^2 + (i-1) x) \ \vert \ x \in {\C}, \ i = 0,1 \}$$ which are realized    as quotients of $$ U_{i} (x) =\mathcal{W}_0. e^{- i\alpha_1 - x (\alpha_1-\alpha_2)}.   $$
	 Morreover, we have:
	 \begin{itemize}
	 \item[(i)] $\dim   L(x, x^2 - x)  _{top} = 1$ for all $x \in {\C}$.
	 \item[(ii)]  $\dim  L(x, x^2 )  _{top} =2 $ for all $x \in {\C}\setminus \{0\}$.
	 \item[(iii)] $U_1 (0) $ is an indecomposable $\mathcal{W}_0$--module. In particular,  $ U_1 (0)_ {top}$ is an indecomposable  $2$-dimensional module for $ A(\mathcal W_0)$.
 	 \end{itemize}
\end{itemize}
\end{theorem}

\begin{proof}
The Clifford vertex algebra $F$ can be embedded into the vertex algebra $V_L$ so that $$\Psi^+ = e^{\alpha_2}, \ \Psi^{-} = e^{-\alpha_2}, \ \omega_F = \frac{1}{2}  \alpha_2 (-1) ^2.$$
Symplectic fermions $\mathcal A(1)$ are also a subalgebra of $V_L$ such that
	$$ b= e^{\alpha_1}, \ c= - \alpha_1(-1) e^{-\alpha_1}, \ \omega_{\mathcal A(1)} = \frac{1}{2} (\alpha_1 (-1) ^2 - \alpha_1(-2) ).   $$
Using the fermionic realization from Theorem \ref{real-ferm} we can now obtain an explicit bosonic realization of $\mathcal{W}_0$.
	
Denote $$ \alpha_1 + \alpha_2  = \gamma, \quad -\alpha_1 - \alpha_2  =  \delta.$$ 

\medskip
Then $\langle \gamma, \delta \rangle = -2$ and from direct calculations we get
	\begin{flalign*}
	G^+(2) G^- 
	&= -3e^{\alpha_1 + \alpha_2}_2 \alpha_1(-1) e^{-\alpha_1 - \alpha_2} &\\
	&= 3 \mathbbm{1} = (k+1)(2k+3)\mathbbm{1},
	\end{flalign*}
	\begin{flalign*}
	G^+(1) G^- 
	&= -3e^{\alpha_1 + \alpha_2}_1 \alpha_1(-1) e^{-\alpha_1 - \alpha_2} &\\
	&= 3 \alpha_2 (-1) = 3(k+1)J,
	\end{flalign*}
	\begin{flalign*}
	G^+(0)G^- 
	&= -3e^{\alpha_1 + \alpha_2}_0 \alpha_1(-1) e^{-\alpha_1 - \alpha_2} &\\
	&= 3J(-1)^2 + 3J(-2) - 3L(-2) = 3J(-1)^2 +(2k+ 3)J(-2) - (k+3)L(-2).
	\end{flalign*}

This concludes the proof of (1).
	
\bigskip	
	
Using  the formula 
\begin{equation} \label{eq: nti produkt}
(\alpha_1(-1)e^{-\alpha_1-\alpha_2})_{n} = \sum_{i=0}^{\infty}(\alpha_1(-i-1)e^{-\alpha_1-\alpha_2}_{n+i}+ e^{-\alpha_1-\alpha_2}_{n-i-1}\alpha_1(i))    
\end{equation} we obtain  for $n \ge 0$, ($i = 0,1$):
\begin{flalign*} 
    G^{-} (n) e^{- i\alpha_1 - x (\alpha_1-\alpha_2)} &= -\sqrt{3} (\alpha_1(-1)e^{-\alpha_1-\alpha_2})_{n+1}  e^{-i\alpha_1 - x (\alpha_1-\alpha_2)} &\\
	 &= -\sqrt{3}\delta_{n,0}e^{-\alpha_1-\alpha_2}_0\alpha_1(0)e^{-i\alpha_1 - x (\alpha_1-\alpha_2)}&\\
	  &= -\sqrt{3}(-i-x)\delta_{n,0}e^{-\alpha_1-\alpha_2}_0e^{-i\alpha_1 - x (\alpha_1-\alpha_2)}&\\
	&= 0,  & \\
	G^{+} (n) e^{- \alpha_1 - x (\alpha_1-\alpha_2)} &= \sqrt{3} e^{\alpha_1+\alpha_2}_n  e^{- \alpha_1  - x (\alpha_1-\alpha_2)}  & \\
	&= \sqrt{3} \delta_{n,0}  e^{ \alpha_2  - x (\alpha_1-\alpha_2)},  & \\
	G^{+} (n) e^{ - x (\alpha_1-\alpha_2)} &= \sqrt{3} e^{\alpha_1+\alpha_2}_n  e^{- x (\alpha_1-\alpha_2)} = 0, & \\
	J(n) e^{- i\alpha_1 - x (\alpha_1-\alpha_2)} &= \alpha_2(n) e^{- i\alpha_1  - x (\alpha_1-\alpha_2)}  & \\
	&= x \delta_{n,0}  e^{- i\alpha_1  - x (\alpha_1-\alpha_2)},  & \\
	L(n) e^{- i\alpha_1 - x (\alpha_1-\alpha_2)} &= \frac{1}{2} \left( \alpha_1 (-1) ^2 - \alpha_1(-2) + \alpha_2(-1) ^2 + \alpha_2(-2)  \right)_n e^{ - i\alpha_1  - x (\alpha_1-\alpha_2)}  &\\
	&= (x^2 + (i-1)x) \delta_{n,0}  e^{ - i\alpha_1  - x (\alpha_1-\alpha_2)}.  &
	\end{flalign*}

This implies that  $e^{- i\alpha_1 - x (\alpha_1-\alpha_2)}$, $i = 0,1$ are highest weight vectors for $\mathcal{W}_0$ and that the highest weight is $ (x, x^2 + (i-1) x)$.  
Next, we claim that
\begin{equation} \label{eq: indecomp}
  G^{-} (n) e^{ \alpha_2 - x (\alpha_1-\alpha_2)} =  \sqrt{3} x \delta_{n,0}  e^{-\alpha_1 - x (\alpha_1-\alpha_2)}.  
\end{equation}

Using the formula (\ref{eq: nti produkt}) again, we have 
\begin{flalign*}
    G^{-} (n) e^{ \alpha_2 - x (\alpha_1-\alpha_2)} &= -\sqrt{3} (\alpha_1(-1)e^{-\alpha_1-\alpha_2})_{n+1}  e^{\alpha_2 - x (\alpha_1-\alpha_2)} &\\
    &= -\sqrt{3}\delta_{n,0}e^{-\alpha_1-\alpha_2}_0\alpha_1(0)e^{\alpha_2 - x (\alpha_1-\alpha_2)}&\\
    &= -\sqrt{3}\delta_{n,0}(-x)e^{-\alpha_1-\alpha_2}_0e^{\alpha_2 - x (\alpha_1-\alpha_2)}&\\
    &= \sqrt{3}x\delta_{n,0}e^{-\alpha_1 - x (\alpha_1-\alpha_2)}. &
\end{flalign*}

From the formula (\ref{eq: indecomp}), and the fact that $h_2(x,y)=0$, we see that $ U_1 (x)_ {top}$ is an indecomposable  $2$-dimensional $A(\mathcal W_0)$--module when $x=0$.
	
Hence $\mathcal{W}_0$ has two families of highest weight modules $U_i(x)$ with highest weights $ (x, x^2 + (i-1) x)$, $i = 0,1$. In particular, their irreducible quotients $L(x, x^2 + (i-1) x)$ are also modules for $\mathcal{W}_0$.
\end{proof}

\appendix

\section{Proof of Lemma  \ref{lema G}}

\begin{lemma}
Let $J = -\frac{1}{3}a_{-1}^{+}a_{-1}^{-}\mathbbm{1}, \; 
	\omega = \frac{1}{2}\left(a_{-2}^{-}a_{-1}^{+}-a_{-2}^{+}a_{-1}^{-}\right)\mathbbm{1}, \; G^+ = \frac{1}{3}\left(a_{-1}^{+}\right)^{3}\mathbbm{1},\;
	G^- = \frac{1}{9}\left(a_{-1}^{-}\right)^{3}\mathbbm{1}.$ Then it holds that:
	\begin{align*} 
	G_{2}^{+}G^{-} & = \frac{2}{9}\mathbbm{1} &\\
	G_{1}^{+}G^{-} & =  -2J &\\
	G_{0}^{+}G^{-} & =  3J_{-1}^{2}-DJ-\frac{4}{3}\omega. &
	\end{align*}
\end{lemma}

\begin{proof}
Let  $\sum_{k=-\infty}^{\infty}c_{k}z^{-k-1} = \left(\sum_{i=-\infty}^{\infty}a_{i}^{+}z^{-i-1}\right)\left(\sum_{j=-\infty}^{\infty}a_{j}^{+}z^{-j-1}\right)$. The coefficients $c_k$ are given by $c_{k+1}=\sum_{l=-\infty}^{\infty}a_{l}^{+}a_{k-l}^{+}$ (since $\left(-l-1\right)+\left(-k+l-1\right)=-\left(k+1\right)-1$).

If we apply this to the formula $$Y\left(\frac{1}{3}\left(a_{-1}^{+}\right)^{3},z\right)=\left(\frac{1}{3}\sum_{n\in\mathbb{Z}}a_{n}^{+}z^{-n-1}\right)^{3}=\frac{1}{3}\sum_{n\in\mathbb{Z}}G_{n}^{+}z^{-n-1},$$ we obtain the following expression for $G_{n}^{+}$: $$ G_{n+1}^{+}=\frac{1}{3}\sum_{m=-\infty}^{\infty}c_{m}a_{n-m}^{+}=\frac{1}{3}\sum_{m=-\infty}^{\infty}\left(\sum_{l=-\infty}^{\infty}a_{l}^{+}a_{m-1-l}^{+}\right)a_{n-m}^{+}. \quad (*)$$ Now we can compute the products $G_{2}^{+}G^{-}$, $G_{1}^{+}G^{-}$ i $G_{0}^{+}G^{-}$.
	
	\begin{itemize}
	\item  $G_{2}^{+}G^{-}$ \\
	From $(*)$ it follows that the only nonzero elements in the product $G_{2}^{+}G^{-}$ are obtained for $ m\geq1$, $m\leq1+l\leq1$, i.e. $m=1$ and $l=1$.
	Since $\left(a_{0}^{+}\right)^{3}\left(a_{-1}^{-}\right)^{3}\mathbbm{1}=6\mathbbm{1}$, we have  $$G_{2}^{+}G^{-}=\frac{1}{27}\left(a_{0}^{+}\right)^{3}\left(a_{-1}^{-}\right)^{3}\mathbbm{1}=\frac{6}{27}\mathbbm{1}=\frac{2}{9}\mathbbm{1}.$$

	\item $G_{1}^{+}G^{-}$ \\
	From $(*)$ it follows that the only nonzero elements in the product $G_{1}^{+}G^{-}$ are obtained for $ m\geq0$, $m\leq1+l\leq1$ $\Rightarrow$ $m=1$ and $l=0$,
	$m=0$ and $l=0$, $m=0$ and $l=-1$.
	Since $\left(a_{0}^{+}\right)^{2}\left(a_{-1}^{-}\right)^{3}\mathbbm{1}=6a_{-1}^{-}\mathbbm{1}$, we have
	\begin{align*}
	G_{1}^{+}G^{-} &= \left(a_{0}^{+}a_{-1}^{+}a_{0}^{+}+a_{-1}^{+}a_{0}^{+}a_{0}^{+}+a_{0}^{+}a_{0}^{+}a_{-1}^{+}\right)\frac{1}{27}\left(a_{-1}^{-}\right)^{3}\mathbbm{1} &\\
	&= \frac{6}{9}a_{-1}^{+}a_{-1}^{-}\mathbbm{1}=-2J. &
	\end{align*}

	\item $G_{0}^{+}G^{-}$ \\
	From $(*)$ it follows that the only nonzero elements in the product $G_{0}^{+}G^{-}$ are obtained for $ m\geq-1$, $m\leq1+l\leq1$ $\Rightarrow$ $m=1$ and $l=0$,
	$m=0$ and $l=0$, $m=0$ and $l=-1$, $m=-1$ and $l=0$, $m=-1$ and $l=-1$,
	$m=-1$ and $l=-2$.
	Since $a_{0}^{+}\left(a_{-1}^{-}\right)^{3}\mathbbm{1}=3\left(a_{-1}^{-}\right)^{2}\mathbbm{1}$, we have
	\begin{align*}
	G_{0}^{+}G^{-} &= (a_{0}^{+}a_{-2}^{+}a_{0}^{+}+a_{-1}^{+}a_{-1}^{+}a_{0}^{+}+a_{-2}^{+}a_{0}^{+}a_{0}^{+}+a_{0}^{+}a_{-1}^{+}a_{-1}^{+}+a_{-1}^{+}a_{0}^{+}a_{-1}^{+} + \\ & +a_{0}^{+}a_{0}^{+}a_{-2}^{+})\frac{1}{27}(a_{-1}^{-})^{3}\mathbbm{1} &\\
	&= \frac{2}{3}a_{-2}^{+}a_{-1}^{-}\mathbbm{1}+\frac{1}{3}\left(a_{-1}^{+}\right)^{2}\left(a_{-1}^{-}\right)^{2}\mathbbm{1}&\\
	&=  3J_{-1}^{2}-DJ-\frac{4}{3}\omega.&
	\end{align*}

\end{itemize}
\end{proof}

\section{Proof of Proposition \ref{zhu_k_5_3}}
Let $\Omega_4$ be the singular vector of conformal weight 4 in $\mathcal{W}^k$ for $k=-5/3$ (cf. Lemma \ref{singW_4}). Using the formula for $\Omega_4$, we outline the computations needed to prove the relation (cf. Proposition \ref{zhu_k_5_3}) $$[G^+]^2([\overline \omega] + \frac{1}{9}) = 0$$ in the Zhu algebra $A_{\overline \omega}(\mathcal{W}_k)$. 


\begin{lemma}
	Let $\Omega_4$ be the singular vector of conformal weight 4 in $\mathcal{W}^k$ for $k=-5/3$. After changing the Virasoro vector to $\overline \omega = \omega + \frac{1}{2} DJ$, singular vector of conformal weight 4 in $\mathcal{W}^k$ is given by
	\begin{align*}
	\overline \Omega_4 &= -\frac{62}{9}  L(-2)^2\mathbbm{1} + \frac{14}{3} L(-4)\mathbbm{1} - 18J(-1)^4\mathbbm{1}  + 31J(-2)J(-1)^2\mathbbm{1}  - & \\
	&- 118J(-3)J(-1)\mathbbm{1}+ \frac{133}{9}J(-2)^2\mathbbm{1} - \frac{8}{9}J(-4)\mathbbm{1} +\frac{62}{9}  L(-2)J(-2)\mathbbm{1} - & \\
	&- 12  L(-3)J(-1)\mathbbm{1}+ 46  L(-2)J(-1)^2\mathbbm{1}  - G^{+}(-2)G^{-}(-2)\mathbbm{1} + & \\
	&+ G^{+}(-1)G^{-}(-3)\mathbbm{1} -18J(-1)G^{+}(-1)G^{-}(-2)\mathbbm{1}. &
	\end{align*}
\end{lemma}

\medskip

Acting with $G^{+}(0)^2 $ on the singular vector $\overline \Omega_4$ yields the following formula:
\begin{flalign*}
	G^{+}(0)^2 \overline \Omega_4 &= -\frac{484}{3}G^{+}(-1)G^{+}(-3)\mathbbm{1} + \frac{704}{9}G^{+}(-2)^2\mathbbm{1} -44J(-1)G^{+}(-1)G^{+}(-2)\mathbbm{1}  + & \\
	&+ 44J(-2)G^{+}(-1)^2\mathbbm{1} + 44 L(-2)G^{+}(-1)^2\mathbbm{1}.&
	\end{flalign*}
	
Recall that the Zhu algebra $A_{\overline \omega}(\mathcal{W}^k)$ is a quotient of the Smith algebra $R(g)$ for $g(x,y) = -(3x^2 - (2k+3)x - (k+3)y)$ (cf. Proposition \ref{smith-zhu}).

\begin{lemma}
Let $E,F,X,Y$ be the generators of the Smith algebra $R(g)$.	In the Zhu algebra $A_{\overline \omega}(\mathcal{W}^k)$ it holds that:
	\begin{enumerate} 
		\item $[G^{+}(-1)G^{+}(-3)\mathbbm{1}] = E^2$
		\item $[G^{+}(-2)^2\mathbbm{1}] = E^2$
		\item $[J(-1)G^{+}(-1)G^{+}(-2)\mathbbm{1}] = -E^2X$
		\item $[J(-2)G^{+}(-1)^2\mathbbm{1}] = -E^2X$
		\item $[ L(-2)G^{+}(-1)^2\mathbbm{1}] = E^2 Y + 2E^2.$
	\end{enumerate}
\end{lemma}

\begin{proof}

From relations	$ (G^+(-2) + G^+(-1) )v \in O(V) $ and $ (G^+(-3) + G^+(-2) )v \in O(V) $ (cf. \cite{Z}), we have that $$ [G^{+}(-3)\mathbbm{1}] =- [G^{+}(-2)\mathbbm{1}] = [G^{+}(-1)\mathbbm{1}] = E.$$ Hence
		\begin{flalign*}
		[G^+(-3)]*[G^+(-1)]= E^2 &= Res_z \frac{(1+z)^0}{z}G^+(z)G^+(-2)\mathbbm{1} + O(V) & \\
		&= Res_z \left(\frac{1}{z} \sum G^+(n)z^{-n-1}G^+(-3)\mathbbm{1}\right) + O(V)& \\
		&= [G^+(-1) G^{+}(-3)\mathbbm{1}]. &
		\end{flalign*}
Also, since	$ (G^+(-1) + G^+(-2) )G^+(-2) \in O(V) $, it follows that
		$$ [G^{+}(-2)^2\mathbbm{1}] =-[G^{+}(-1)G^{+}(-2)\mathbbm{1}] = E^2. $$
Next, we compute:	
		\begin{flalign*}
		[G^{+}(-1)G^+(-2)]*[J(-1)]=-E^2X &= Res_z \frac{(1+z)^0}{z}J(z)G^{+}(-1)G^+(-2)\mathbbm{1} + O(V) & \\
		&= Res_z \left(\frac{1}{z} \sum J(n)z^{-n-1}G^{+}(-1)G^+(-2)\mathbbm{1}\right) + O(V)& \\
		&= [J(-1) G^{+}(-1)G^{+}(-2)\mathbbm{1}], &
		\end{flalign*}
		\begin{flalign*}
		[J(-2)G^+(-1)]*[G^+(-1)] &=(-XE+E)E = Res_z \frac{(1+z)^0}{z}G^+(z)J(-2)G^+(-1)\mathbbm{1} + O(V) & \\
		&= [G^+(-1) J(-2)G^+(-1)\mathbbm{1}] & \\
		&= -[G^{+}(-3)G^+(-1)\mathbbm{1}]+ [J(-2)G^{+}(-1)^2\mathbbm{1}] &
		\end{flalign*}
		
		$\Longrightarrow  [J(-2)G^{+}(-1)^2\mathbbm{1}] = XE^2 + 2E^2 = -E^2X, $
		\begin{flalign*}
		[G^{+}(-1)^2]*[ L(-2)]= E  Y &= Res_z \frac{(1+z)^1}{z}  L(z)G^{+}(-1)^2\mathbbm{1} + O(V) & \\
		&= [ L(-2)G^{+}(-1)^2\mathbbm{1}] + [ L(-1)G^{+}(-1)^2\mathbbm{1}] & \\
		&= [ L(-2)G^{+}(-1)\mathbbm{1}] + 2[G^{+}(-1)G^{+}(-2)\mathbbm{1}] &
		\end{flalign*}
		$\Longrightarrow  [ L(-2)G^{+}(-1)^2\mathbbm{1}] = E^2 Y + 2E^2. $

\end{proof}

\begin {proposition} 
In the Zhu algebra $A_{\overline \omega}(\mathcal{W}_k)$ it holds that $$[G^+]^2([\overline \omega] + \frac{1}{9}) = 0.$$

\end{proposition}

\begin{proof}
From the above computation we have that
	\begin{flalign*}
	[G^{+}(0)^2 \overline \Omega_4] &= -\frac{484}{3}E^2 + \frac{704}{9}E^2 -44(-E^2X)  + 44(-E^2X) +44(E^2  Y + 2E^2)& \\
	&= 44E^2( Y + \frac{1}{9}),&
	\end{flalign*}
	and the claim follows.
\end{proof}

\section*{Acknowledgment}
 
This paper is based in part on A.K.   Ph.D. dissertation \cite{K-Phd}.
 The results of this paper were also  reported by A. K. at the conferences Vertex algebras and infinite dimensional Lie algebras,  Split, November 22-25, 2018 and Representation Theory XI, Dubrovnik, June 23-29.2019.

The authors would like to thank to  T. Arakawa,   D. Ridout,  A. Linshaw for valuable discussions.  
The authors  are  partially supported   by the
QuantiXLie Centre of Excellence, a project cofinanced
by the Croatian Government and European Union
through the European Regional Development Fund - the
Competitiveness and Cohesion Operational Programme
(KK.01.1.1.01.0004).

\Addresses

\end{document}